\documentclass[11pt,a4paper]{article}

\usepackage{amsmath,amsfonts,amssymb}
\usepackage{bbm}
\usepackage{cases}

\usepackage{lipsum}
\usepackage{subcaption}

\usepackage{xcolor}
\usepackage{graphicx}
\usepackage{ulem}

\newtheorem{defn}{Definition}[section]
\newtheorem{theo}[defn]{Theorem}
\newtheorem{lem}[defn]{Lemma}
\newtheorem{prop}[defn]{Proposition}
\newtheorem{cor}[defn]{Corollary}
\newtheorem{rem}[defn]{Remark}
\newtheorem{exam}[defn]{Example}

\newenvironment{proof}{{\bf Proof }}{{\vskip 0.1cm \hfill$\Box$}}

\begin{document} 

\noindent
{\Large \bf Well-posedness of linear elliptic equations with $L^d$-drifts under divergence-type conditions}
\\ \\
\bigskip
\noindent
{\bf Haesung Lee}  \\
\noindent
{\bf Abstract.}  
We establish the well-posedness of linear elliptic equations with critical-order drifts in $L^d$ and positive zero-order coefficients in $L^1$ or $L^{\frac{2d}{d+2}}$, where classical methods are often too restrictive. Our approach relies on a divergence-free transformation and a structural condition on the drift vector field, which admits a decomposition into a regular component and another whose weak divergence belongs to $L^{\tilde{q}}$ for some $\tilde{q} > \frac{d}{2}$. This condition is essential for constructing a suitable weight function $\rho$ via the weak maximum principle and the Harnack inequality. Within this framework, we prove the existence and uniqueness of weak solutions, significantly relaxing the regularity assumptions on the zero-order coefficients in $L^{\frac{d}{2}}$.


\section{Introduction} \label{intro}
 
This paper establishes the well-posedness (existence and uniqueness of weak solutions) (cf. Definition \ref{weaksoldefinition}) of the following Dirichlet problem for a linear elliptic equation in divergence form, defined on a bounded open subset $U \subset \mathbb{R}^d$ with $d \geq 3$:
\begin{equation} \label{maineq2}
\left\{
\begin{alignedat}{2}
- \operatorname{div}(A \nabla u) + \langle \mathbf{H}, \nabla u \rangle + c u &= f && \quad \text{in } U, \\
u &= 0 && \quad \text{on } \partial U,
\end{alignedat}
\right.
\end{equation}
where $A$ is uniformly strictly elliptic and bounded on $U$ (see \eqref{elliptici}).
The well-posedness of \eqref{maineq2} was established in \cite{L25jm} not by the conventional bilinear form methods, but by employing weak convergence techniques combined with a divergence-free transformation, under the assumptions that $\mathbf{H} \in L^p(U, \mathbb{R}^d)$ with $p \in (d, \infty)$, $c \in L^1(U)$, and $f \in L^{q}(U)$ with $q \in (\frac{d}{2}, \infty)$. It is therefore a natural problem to investigate whether the condition $\mathbf{H} \in L^p(U, \mathbb{R}^d)$ can be relaxed to the critical case $\mathbf{H} \in L^d(U, \mathbb{R}^d)$. However, under the sole assumption $\mathbf{H} \in L^d(U, \mathbb{R}^d)$, an extension of the result to the cases $c \in L^1(U)$ or $c \in L^{\frac{2d}{d+2}}(U)$ cannot be achieved directly by either the standard bilinear form methods or the approach in \cite{L25jm}.\\
The well-posedness of \eqref{maineq2} via bilinear form methods based on the Lax-Milgram theorem originates from G. Stampacchia's work \cite{S65}, where it was proved that there exists a constant $\gamma > 0$, depending only on $A$, $\mathbf{H}$, and $d$, such that the problem \eqref{maineq2} admits a unique weak solution whenever $c \geq \gamma$ and $c \in L^\frac{d}{2}(U)$ with $d \geq 3$. Similar results, under certain restrictions on the zero-order coefficients, are treated in \cite{LU68}.
In \cite{T73} (cf. \cite[Section 8.2]{GT}), N. S. Trudinger established the well-posedness of \eqref{maineq2} by developing a weak maximum principle under the assumptions $\mathbf{H} \in L^d(U, \mathbb{R}^d)$, $c \in L^{\frac{d}{2}}(U)$ with $c \geq 0$, and $f \in L^{\frac{2d}{d+2}}(U)$.
In the absence of the classical coercivity property, the well-posedness of \eqref{maineq2} has also been obtained via a duality argument in \cite{Dr}. For another reference on non-coercive linear equations with coefficients in Lorentz spaces, \cite{M11} establishes the well-posedness of the dual problem associated with \eqref{maineq2}. For further results beyond the $L^2(U)$-regularity of $\nabla u$, or for corresponding results concerning non-divergence type counterparts of \eqref{maineq2}, we refer to \cite{L25jm} and references therein. \\
To understand the technical challenge in the critical case of the drift coefficients, we first revisit the approach of \cite{L25jm} under the assumptions $\mathbf{H} \in L^p(U, \mathbb{R}^d)$ for some $p \in (d, \infty)$ and $c \in L^1(U)$ with $c \geq 0$.
In \cite{L25jm}, to apply a divergence-free transformation, one first constructs a strictly positive function $\rho \in H^{1,2}(U) \cap C(\overline{U})$ and a divergence-free vector field $\rho \mathbf{B} \in L^2(U, \mathbb{R}^d)$, which then transforms the original equation \eqref{maineq2} into the form shown in \eqref{maineq3} (see Theorem \ref{divfretransform}):
\begin{equation} \label{maineq3}
\left\{
\begin{alignedat}{2}
	- \operatorname{div}(\rho A \nabla u) + \langle \rho \mathbf{B}, \nabla u \rangle + \rho c u &= \rho f && \quad \text{in } U, \\
u &= 0 && \quad \text{on } \partial U.
\end{alignedat}
\right.
\end{equation}
This reformulation enables the derivation of a priori $H^{1,2}$ and $L^{\infty}$-bounds, and through a delicate application of weak compactness methods and a duality argument, one can obtain the existence and uniqueness of a weak solution $u \in H^{1,2}_0(U)$ to \eqref{maineq2}.\\
In the critical case where $\mathbf{H} \in L^d(U, \mathbb{R}^d)$, the lack of regularity makes it difficult to construct the function $\rho \in H^{1,2}(U) \cap C(\overline{U})$ and the vector field $\rho \mathbf{B}$ as in the approach described above. 
The first reason is that under the condition $\mathbf{H} \in L^d(U, \mathbb{R}^d)$, obtaining the key estimate in Lemma \ref{lempactprelim} may not be directly derived but requires delicate computations and the use of a partition of unity.
The second reason is that the method employed in \cite{L25jm} relies on the construction of a function $\rho$ satisfying an elliptic Harnack inequality and H\"{o}lder continuity. However, such a construction is significantly restricted under the assumption $\mathbf{H} \in L^d(U, \mathbb{R}^d)$ (see Remark~\ref{remcounter}), which justifies the necessity of imposing additional conditions on $\mathbf{H}$ in $L^d(U, \mathbb{R}^d)$. Indeed, we show that the divergence-free transformation can still be successfully carried out if $\mathbf{H}$ admits a decomposition where one component has a sufficiently regular divergence in $L^{\tilde{q}}$ for some $\tilde{q} \in \left( \frac{d}{2}, \infty \right)$.\\  \\
Before stating our main result, let us present the main assumption in this paper: \\ \\
{\bf (T)}
{\it
$U$ is a bounded open subset of $\mathbb{R}^d$ with $d \geq 3$, and $B_r(x_0)$ is an open ball in $\mathbb{R}^d$ such that $\overline{U} \subset B_r(x_0)$.
$\mathbf{H}_1 \in L^p(B_r(x_0), \mathbb{R}^d)$ for some $p \in (d, \infty)$, and $\mathbf{H}_2 \in L^d(B_r(x_0), \mathbb{R}^d)$ satisfies the following distributional identity (see Definition \ref{defnweaksol}): there exists $\tilde{h} \in L^{\tilde{q}}(B_r(x_0))$ with $\tilde{q} \in (\frac{d}{2}, \infty)$ such that
$$
\int_{B_r(x_0)} \langle \mathbf{H}_2, \nabla \psi \rangle \, dx = -\int_{B_r(x_0)} \tilde{h} \psi \, dx \quad \text{for all } \psi \in C_0^{\infty}(B_r(x_0)),
$$
i.e., $\operatorname{div} \mathbf{H}_2=\tilde{h} \in L^{\tilde{q}}(B_r(x_0))$. $\mathbf{H} := \mathbf{H}_1 + \mathbf{H}_2 \in L^d(B_r(x_0), \mathbb{R}^d)$. $A = (a_{ij})_{1 \leq i,j \leq d}$ is a (possibly non-symmetric) matrix of measurable functions on $\mathbb{R}^d$ such that there exist constants $M > 0$ and $\lambda > 0$~satisfying
\begin{equation} \label{elliptici}
\max_{1 \leq i,j \leq d} |a_{ij}(x)| \leq M, \quad \langle A(x) \xi, \xi \rangle \geq \lambda \| \xi \|^2 \quad \text{for a.e. } x \in \mathbb{R}^d \text{ and all } \xi \in \mathbb{R}^d.
\end{equation}
}
\\
The following is the main theorem of this paper, which shows that the conclusion of \cite[Theorem 1.1]{L25jm} remains robust under the assumption of an $L^d$-drift.
\begin{theo} \label{mainexisthm}
Assume that {\bf (T)} holds. Let $c \in L^1(U)$ with $c \geq 0$ in $U$. Then, the following statements~hold:
\begin{itemize}
\item[(i)]
Let $v \in H_0^{1,2}(U)$ with $cv \in L^1(U)$ be such that
\begin{equation} \label{weakdzero}
\int_U \langle A \nabla v, \nabla \psi \rangle + \langle \mathbf{H}, \nabla v \rangle \psi + cv \psi \, dx = 0 \quad \text{for all } \psi \in C_0^{\infty}(U).
\end{equation}
Then, $v = 0$.

\item[(ii)]
Let $f \in L^q(U)$ for some $q > \frac{d}{2}$. Then, there exists a unique weak solution $u \in H_0^{1,2}(U) \cap L^{\infty}(U)$ to \eqref{maineq2}. Moreover, $u$ satisfies
\begin{equation} \label{estimenbd2}
\|u\|_{H_0^{1,2}(U)} \leq K_5 \|f\|_{L^{\frac{2d}{d+2}}(U)},
\end{equation}
and
\begin{equation} \label{bddestmizte}
\|u\|_{L^{\infty}(U)} \leq K_6 \|f\|_{L^q(U)},
\end{equation}
where $K_5 = \tilde{K}_1 K_3$, $K_6 = \tilde{K}_1 K_4$, $\tilde{K}_1 > 0$ is a constant as in Theorem \ref{existrho}, depending only on $d$, $\lambda$, $M$, $B_r(x_0)$, $p$, $\tilde{q}$, $\mathbf{H}$, the constant $K_3>0$ depends only on $d$, $\frac{\lambda}{\tilde{K}_1}$, $|U|$, and the constant $K_4 > 0$ depends only on $d$, $\frac{\lambda}{\tilde{K}_1}$, $|U|$, $q$.
In particular, if $\alpha > 0$ is a constant, $c \geq \alpha$ in $U$, and $f \in L^{\theta}(U) \cap L^q(U)$ with $\theta \in [1, \infty]$, then $u$ satisfies the following contraction estimate:
\begin{equation} \label{normalcontrac}
\|u\|_{L^{\theta}(U)} \leq \frac{\tilde{K}_1}{\alpha} \|f\|_{L^{\theta}(U)}.
\end{equation}

\item[(iii)]
Let $c \in L^{\frac{2d}{d+2}}(U)$ and $f \in L^{\frac{2d}{d+2}}(U)$. Then there exists a unique weak solution $u \in H_0^{1,2}(U)$ to \eqref{maineq2}. Moreover, $u$ satisfies \eqref{estimenbd2}.
In particular, if $\alpha > 0$ is a constant, $c \geq \alpha$ in $U$, and $f \in L^{\theta}(U) \cap L^{\frac{2d}{d+2}}(U)$ with $\theta \in [1, \infty]$, then $u$ satisfies \eqref{normalcontrac}.
\end{itemize}
\end{theo}
Although the main result of the paper concerns the existence and uniqueness of weak solutions, it is particularly noteworthy that no additional structural conditions, such as the $VMO$ assumption on the matrix of functions $A$, are imposed. Furthermore, the main result Theorem \ref{mainexisthm} challenges the conventional belief that the optimal regularity condition for the zero-order term is $c \in L^{\frac{d}{2}}(U)$ by demonstrating that the weaker assumptions $c \in L^1(U)$ or even $c \in L^{\frac{2d}{d+2}}(U)$ are sufficient.
This observation suggests the possibility of further developments that could partially relax the regularity condition on the zero-order coefficient, namely $c \in L^{\frac{d}{2}}(U)$, in order to obtain the well-posedness results for non-divergence form equations as established in \cite{V92, K21}. \\
This paper is organized as follows:
Section \ref{sec2} introduces the essential notations and definitions used throughout the paper.
Section \ref{sec3} establishes fundamental inequalities and shows Lemma \ref{mainkeyprese}, a crucial inequality for this paper.
Section \ref{sec4} sets up the divergence-free transformation and completes the proof of Theorem \ref{mainexisthm}.
Section \ref{sec5} concludes with a discussion.

\section{Notations and definitions} \label{sec2}

In this paper, we work within the $d$-dimensional Euclidean space $\mathbb{R}^d$, where $d \geq 1$, equipped with the standard inner product $\langle \cdot, \cdot \rangle$ and the corresponding Euclidean norm $\| \cdot \|$. For a point $x_0 \in \mathbb{R}^d$ and radius $r > 0$, we denote by $B_r(x_0)$ the open ball $\{ x \in \mathbb{R}^d : \|x - x_0\| < r \}$, and we write $B_r:=\{ x \in \mathbb{R}^d : \|x\| < r \}$. The Lebesgue measure on $\mathbb{R}^d$ is denoted by $dx$, and for a measurable set $E \subset \mathbb{R}^d$, $dx(E)$ is written as $|E|$. The indicator function of a set $W$ is denoted by $1_W$. Let $U$ be an open subset of $\mathbb{R}^d$. We denote by $C(U)$ and $C(\overline{U})$ the spaces of continuous functions on $U$ and its closure $\overline{U}$, respectively. For $k \in \mathbb{N} \cup \{\infty\}$, the space $C^k(U)$ consists of functions that are $k$-times continuously differentiable on $U$, while $C_0^k(U)$ denotes the subspace of $C^k(U)$ consisting of functions with compact support in $U$. Let $s \in [1, \infty]$. We denote by $L^s(U)$ the standard Lebesgue space with norm $\|\cdot\|_{L^s(U)}$, and by $L^s(U, \mathbb{R}^d)$ the space of $\mathbb{R}^d$-valued functions whose components lie in $L^s(U)$, equipped with the norm $\|\mathbf{F}\|_{L^s(U)}:=\| \| \mathbf{F} \| \|_{L^s(U)}$. For each $i \in \{1,2, \ldots d \}$, $\partial_i$ denotes the weak partial derivative with respect to the $i$-th component. The weak gradient of a function $u$ is denoted by $\nabla u:=(\partial_1 u, \ldots, \partial_d u)$. The Sobolev space $H^{1,s}(U)$ consists of functions in $L^s(U)$ whose weak partial derivatives also belong to $L^s(U)$. The space $H_0^{1,2}(U)$ denotes the closure of $C_0^\infty(U)$ in the $H^{1,2}(U)$-norm.  By the Poincaré inequality, we write $\|u \|_{H^{1,2}_0(U)}:= \| \nabla u\|_{L^2(U)}$. The dual space of $H_0^{1,2}(U)$ is denoted by $H^{-1,2}(U)$, and the duality pairing is represented by $\langle \cdot, \cdot \rangle_{H^{-1,2}(U)}$.

\begin{defn} \label{weaksoldefinition}
Let $A = (a_{ij})_{1 \leq i,j \leq d}$ be a matrix of bounded and measurable functions on $\mathbb{R}^d$. Let $\mathbf{H} \in L^2(U, \mathbb{R}^d)$, $c \in L^1(U)$, and $f \in L^1(U)$. We say that $u$ is a weak solution to \eqref{maineq2} if $u \in H_0^{1,2}(U)$ and $c u \in L^1(U)$, and the following identity holds:
\begin{equation} \label{maineqfir}
\int_U \langle A \nabla u, \nabla \psi \rangle + \langle \mathbf{H}, \nabla u \rangle \psi + c u \psi \, dx 
= \int_U f \psi \, dx 
\quad \text{for all } \psi \in C_0^\infty(U).
\end{equation}
\end{defn}

\begin{defn} \label{defnweaksol}
For a vector field $\mathbf{H} \in L^1_{\mathrm{loc}}(U, \mathbb{R}^d)$, its divergence $\operatorname{div} \mathbf{H}$ is understood in the weak sense. That is, if $h \in L^1_{\mathrm{loc}}(U)$ satisfies
\[
\int_U \langle \mathbf{H}, \nabla \varphi \rangle \, dx = -\int_U h \varphi \, dx \quad \text{for all } \varphi \in C_0^{\infty}(U),
\]
we write $\operatorname{div} \mathbf{H} = h$ in $U$. A vector field $\mathbf{H}$ is called divergence-free if $\operatorname{div} \mathbf{H} = 0$.
\end{defn}

\section{Fundamental inequalities} \label{sec3}

In this section, we mainly assume the condition {\bf (T1)} below, which is weaker than {\bf (T)}.\\ \\
{\bf (T1)}:
{\it
$U$ is a bounded open subset of $\mathbb{R}^d$ with $d \geq 3$, $\mathbf{H} \in L^d(U, \mathbb{R}^d)$, and $A = (a_{ij})_{1 \leq i,j \leq d}$ is a (possibly non-symmetric) matrix of measurable functions on $\mathbb{R}^d$ such that there exist constants $M > 0$ and $\lambda > 0$ satisfying \eqref{elliptici}.
}
\begin{prop} \label{bilineaprop}
Under the assumption {\bf (T1)}, define a bilinear form $\mathcal{B}: H^{1,2}_0(U) \times H^{1,2}_0(U) \rightarrow \mathbb{R}$ given~by
\begin{equation} \label{bilinearf}
\mathcal{B}(f,g):= \int_{U} \langle A \nabla f + f \mathbf{H}, \nabla g \rangle \,dx, \quad f,g \in H^{1,2}_0(U).
\end{equation}
Then, the following statements hold:
\begin{itemize}
\item[(i)]
\begin{equation} \label{desiredestim1}
|\mathcal{B}(f,g)| \leq K \| \nabla f \|_{L^2(U)}  \| \nabla g \|_{L^2(U)} \quad \text{ for all $f,g \in H^{1,2}_0(U)$},
\end{equation}
where $K:=dM + \dfrac{2(d-1)}{d-2} \|\mathbf{H} \|_{L^d(U)}$.
\item[(ii)]
Let $N \geq 0$ be a constant such that
\begin{equation} \label{vectorfieh}
\left( \int_U 1_{\{ \|\mathbf{H}\| \geq N \}} \| \mathbf{H} \|^d \, dx \right)^{\frac{2}{d}} 
\leq \frac{\lambda^2}{16} \left( \frac{d-2}{d-1} \right)^2.
\end{equation}
Then,
\begin{equation} \label{desiredestim2}
\mathcal{B}(f,f) +\frac{N^2}{\lambda} \| f\|^2_{L^2(U)}  \geq \frac{\lambda}{2}\|\nabla f \|^2_{L^2(U)} \quad \text{ for all $f \in H^{1,2}_0(U)$}.
\end{equation}
\end{itemize}
\end{prop}
\begin{proof}
(i) Let $f, g \in H^{1,2}_0(U)$.  
By the Sobolev inequality \cite[Section 5.6, Theorem 1]{E10} (cf. \cite[Theorem 4.8]{EG15}), we obtain
\begin{equation} \label{imptsobolmud}
\|g\|_{L^{\frac{2d}{d-2}}(U)} \leq \frac{2(d-1)}{d-2} \|\nabla g\|_{L^2(U)}.
\end{equation}
Applying the Cauchy–Schwarz inequality and the H\"{o}lder inequality, we have
\begin{align*}
\left| \int_U \langle f\mathbf{H}, \nabla g \rangle  \, dx \right|
&\leq \|\mathbf{H}\|_{L^d(U)} \|f\|_{L^{\frac{2d}{d-2}}(U)} \|\nabla g\|_{L^2(U)}  \\
&\leq \frac{2(d-1)}{d-2} \|\mathbf{H}\|_{L^d(U)} \|\nabla f\|_{L^2(U)} \|\nabla g\|_{L^2(U)}.
\end{align*}
In addition, we estimate
\[
\left| \int_U \langle A \nabla f, \nabla g \rangle \, dx \right|
\leq dM \int_U \|\nabla f\| \, \|\nabla g\| \, dx
\leq dM \|\nabla f\|_{L^2(U)} \|\nabla g\|_{L^2(U)}.
\]
Hence, the desired estimate \eqref{desiredestim1} follows. \\ \\
(ii) Let $f \in H^{1,2}_0(U)$.  
Using the Cauchy–Schwarz inequality and Young's inequality, we obtain
\begin{align} \label{driftestimd2}
\left| \int_{U} \langle \mathbf{H}, \nabla f \rangle f \, dx \right|
&\leq \int_U \|\mathbf{H}\|\, |f|\, \|\nabla f\| \, dx \notag \\
&\leq \frac{\lambda}{4} \int_U \|\nabla f\|^2 dx + \frac{1}{\lambda} \int_U \|\mathbf{H}\|^2 |f|^2 dx.
\end{align}
Define the function $\phi: [0, \infty) \to \mathbb{R}$ by
\[
\phi(s) := \left( \int_U 1_{\{ \|\mathbf{H}\| \geq s \}} \|\mathbf{H}\|^d dx \right)^{\frac{2}{d}}, \quad s \in [0, \infty).
\]
Since $\|\mathbf{H}\| \in L^d(U)$ and is finite almost everywhere, it follows from the Lebesgue dominated convergence theorem that
\[
\lim_{s \to \infty} \phi(s) = 0.
\]
Now choose $N \geq 0$ such that
\begin{equation} \label{imptindex2}
\phi(N) \leq \frac{\lambda^2}{16} \left( \frac{d-2}{d-1} \right)^2.
\end{equation}
Using again the H\"{o}lder inequality, \eqref{imptindex2}, and the Sobolev inequality \eqref{imptsobolmud}, we estimate
\begin{align*}
\int_U \|\mathbf{H}\|^2 |f|^2 dx
&= \int_U 1_{\{ \|\mathbf{H}\| \geq N \}} \|\mathbf{H}\|^2 |f|^2 dx + \int_U 1_{\{ \|\mathbf{H}\| < N \}} \|\mathbf{H}\|^2 |f|^2 dx \\
&\leq \left( \int_U 1_{\{ \|\mathbf{H}\| \geq N \}} \|\mathbf{H}\|^d dx \right)^{\frac{2}{d}} \|f\|_{L^{\frac{2d}{d-2}}(U)}^2 + N^2 \int_U |f|^2 dx \\
&\leq \frac{\lambda^2}{4} \|\nabla f\|_{L^2(U)}^2 + N^2 \int_U |f|^2 dx.
\end{align*}
Substituting this into \eqref{driftestimd2}, we obtain
\begin{align*}
\int_U \langle \mathbf{H}, \nabla f \rangle f \, dx
&\geq -\frac{\lambda}{2} \int_U \|\nabla f\|^2 dx - \frac{N^2}{\lambda} \int_U |f|^2 dx.
\end{align*}
Since
\[
\int_U \langle A \nabla f, \nabla f \rangle dx \geq \lambda \int_U \|\nabla f\|^2 dx,
\]
the desired estimate \eqref{desiredestim2} follows.
\end{proof}
\centerline{}
The following existence result is well known and can be found in \cite{BKRS15,GT,S65}. For clarity and the reader's convenience, we state the details here.
\begin{prop} \label{laxfredexist}
Assume {\bf (T1)}. Let $\mathcal{B}: H^{1,2}_0(U) \times H^{1,2}_0(U) \to \mathbb{R}$ denote the bilinear form defined by \eqref{bilinearf}. Suppose that if $w \in H^{1,2}_0(U)$ satisfies
$$
\mathcal{B}(w, \varphi)=0 \quad \text{ for all $\varphi \in H^{1,2}_0(U)$},
$$
then $w=0$ in $U$. Then, for each $\psi \in H^{-1,2}(U)$, there exists a unique $u_{\psi} \in H^{1,2}_0(U)$ such that
$$
\mathcal{B}(u_{\psi}, \varphi) = \langle \psi, \varphi \rangle_{H^{-1,2}(U)} \quad \text{ for all $\varphi \in H^{1,2}_0(U)$}.
$$
\end{prop}
\begin{proof}
Let $N \geq 0$ be the constant appearing in \eqref{vectorfieh}, and define $\gamma := \frac{N^2}{\lambda}$.
Define a bilinear form $\mathcal{B}_{\gamma}: H^{1,2}_0(U) \times H^{1,2}_0(U) \rightarrow \mathbb{R}$ given by
$$
\mathcal{B}_{\gamma}(f,g):=\mathcal{B}(f,g) + \gamma \int_U f g dx, \quad  f,g \in H^{1,2}_0(U).
$$
By the Lax–Milgram theorem (cf. \cite[Corollary 5.8]{Br11}) and Proposition \ref{bilineaprop}, for each $\psi \in H^{-1,2}(U)$, there exists a unique $u_{\gamma, \psi} \in H^{1,2}_0(U)$ such that
\begin{equation} \label{variationiden}
\mathcal{B}_{\gamma}(u_{\gamma, \psi}, \varphi) = \langle \psi, \varphi \rangle_{H^{-1,2}(U)} \quad \text{ for all $\varphi \in H^{1,2}_0(U)$}.
\end{equation}
Substituting $\varphi = u_{\gamma, \psi}$ into \eqref{variationiden} and applying Proposition~\ref{bilineaprop}(ii), we obtain
\begin{align*}
\frac{\lambda}{2} \| \nabla u_{\gamma, \psi} \|^2_{L^2(U)} \leq \mathcal{B}_{\gamma}(u_{\gamma, \psi}, u_{\gamma, \psi}) = \langle \psi, u_{\gamma, \psi} \rangle_{H^{-1,2}(U)} \leq \|\psi \|_{H^{-1,2}(U)} \| \nabla u_{\gamma, \psi} \|_{L^{2}(U)},
\end{align*}
and hence,
\begin{equation} \label{energyestima}
\| \nabla u_{\gamma, \psi} \|_{L^2(U)} \leq \frac{2}{\lambda} \| \psi\|_{H^{-1,2}(U)}.
\end{equation}
Define $K:H^{-1,2}(U) \rightarrow H^{1,2}_0(U)$ given by
\begin{equation} \label{defnofk}
K \psi:= u_{\gamma, \psi}, \quad \psi \in H^{-1,2}(U).
\end{equation}
Then, by \eqref{energyestima}, we have
$$
\| K \psi \|_{H^{1,2}_0(U)} \leq \frac{2}{\lambda} \| \psi\|_{H^{-1,2}(U)} \quad \text{ for all } \psi \in H^{-1,2}(U).
$$
Define the operator $J: H^{1,2}_0(U) \to H^{-1,2}(U)$ by, for each $u \in H^{1,2}_0(U)$,
\begin{equation} \label{inclusionmap}
\langle J(u), \varphi \rangle_{H^{-1,2}(U)} = \int_{U} u \varphi \,dx  \quad \text{ for all $\varphi \in H^{1,2}_0(U)$}.
\end{equation}
By the Rellich-Kondrachov compactness theorem, $J$ is a compact operator, and hence so is $K \circ J$. We now state the following claim. \\ \\
{\underline{\sf Claim}:} Let $u \in H_0^{1,2}(U)$ and $\psi \in H^{-1,2}(U)$. Then, the following statements {\rm (a)--(b)} are equivalent:
\begin{equation} \label{bvpequival}
\left\{
\begin{aligned}
    u- \gamma \big( K \circ J  \big)u &= K \psi \quad \text{ in } H^{1,2}_0(U). &&  \;\;\qquad {\rm (a)} \\[1em]
    \mathcal{B}(u, \varphi)&= \langle \psi, \varphi \rangle_{H^{-1,2}(U)} \qquad \text{ for all $\varphi \in H^{1,2}_0(U)$.} && \;\;\qquad {\rm (b)}
\end{aligned}
\right.
\end{equation}
To prove the claim, first suppose that {\rm (a)} holds. Then, we have
\[
u = K \big( \psi + \gamma J(u) \big).
\]
By the definition of $K$ (see \eqref{defnofk} and \eqref{variationiden}), it follows that
\begin{equation} \label{variatcondia}
\mathcal{B}_{\gamma}(u, \varphi) = \langle \psi + \gamma J(u), \varphi \rangle_{H^{-1,2}(U)} \quad \text{for all } \varphi \in H^{1,2}_0(U),
\end{equation}
which is equivalent to {\rm (b)} by \eqref{inclusionmap}.
Conversely, assume that {\rm (b)} holds. Then, \eqref{variatcondia} is satisfied, and hence, by the definition of $K$, we obtain
\[
u = K \big( \psi + \gamma J(u) \big),
\]
which implies {\rm (a)} by \eqref{inclusionmap}. This completes the proof of the claim.  \\ \\
Let $I: H^{1,2}_0(U) \to H^{1,2}_0(U)$ denote the identity operator. Evaluating {\rm (a)} with $\psi = 0$, it follows from  the equivalence established in \eqref{bvpequival} that
\[
\text{$u \in H^{1,2}_0(U)$ satisfies } \big(I - \gamma (K \circ J) \big) u = 0
\]
if and only if
\[
\mathcal{B}(u, \varphi) = 0 \quad \text{for all } \varphi \in H^{1,2}_0(U).
\]
Thus, by assumption, it follows that
\[
\left\{ u \in H^{1,2}_0(U) : \big(I - \gamma (K \circ J)\big) u = 0 \right\} = \{0\}.
\]
Since $\gamma (K \circ J): H^{1,2}_0(U) \to H^{1,2}_0(U)$ is a compact operator, the Fredholm alternative (see \cite[Theorem~6.6]{Br11}) implies that for each $\psi \in H^{-1,2}(U)$, there exists $u_{\psi} \in H^{1,2}_0(U)$ such that
\[
\big(I - \gamma (K \circ J)\big) u_\psi = K \psi.
\]
Therefore, by the equivalence established in \eqref{bvpequival}, the desired assertion follows.
\end{proof}

\begin{prop}  {\bf  (Poincaré-type inequality)}
\label{mainpoincare}
The following inequality holds:
\begin{align*} 
\|f\|_{L^2(U)} \leq \frac{2(d-1)}{d} |U|^{\frac{1}{d}} \| \nabla f\|_{L^2(U)} \quad \text{for all } f \in H^{1,2}_0(U).
\end{align*}
\end{prop}
\begin{proof}
By applying the Gagliardo-Nirenberg-Sobolev inequality (\cite[Section 5.6, Theorem 1]{E10}) together with the H\"{o}lder inequality, we obtain
\begin{align*}
\|f\|_{L^2(U)} \leq \frac{\frac{2d}{d+2}(d-1)}{d - \frac{2d}{d+2}} \| \nabla f \|_{L^{\frac{2d}{d+2}}(U)} 
\leq \frac{2(d-1)}{d} |U|^{\frac{1}{d}} \| \nabla f\|_{L^2(U)} \quad \text{for all } f \in H^{1,2}_0(U).
\end{align*}
\end{proof}

\begin{lem} \label{basicapplem}
Let $\phi \in C^1((-\varepsilon, \infty))$ with $\varepsilon>0$ be such that $\phi(0)=0$ and $\phi' \in L^{\infty}((0, \infty))$. If $v \in H^{1,2}_0(U)$ with $v \geq 0$ in $U$, then $\phi(v) \in H^{1,2}_0(U)$ and $\nabla \phi(v) = \phi'(v) \nabla v$ in $U$.
\end{lem}
\begin{proof}
Extend $\phi$ to a function on $\mathbb{R}$, denoted again by $\phi$, such that $\phi \in C^1(\mathbb{R})$ with $\phi' \in L^{\infty}(\mathbb{R})$.
Let $(v_n)_{n \geq 1} \subset C_0^{\infty}(U)$ be a sequence of functions such that $\lim_{n \rightarrow \infty} v_n = v$ in $H^{1,2}_0(U)$ such that
\begin{equation} \label{h12conver}
\lim_{n \to \infty} \| \nabla v_n - \nabla v \|_{L^2(U)} = 0
\end{equation}
and
\begin{equation} \label{pointconver}
\lim_{n \to \infty} v_n(x) = v(x) \quad \text{for a.e. } x \in U.
\end{equation}
Then, by the chain rule, $(\phi(v_n))_{n \geq 1} \subset C^1_0(U)$ and
$$
\nabla \phi(v_n) = \phi'(v_n) \nabla v_n \quad \text{ in $U$ \quad for each $n \geq 1$}.
$$
Moreover, by \cite[Theorem 4.4(ii)]{EG15}, $\phi(v) \in H^{1,2}(U)$ satisfies $\nabla \phi(v) = \phi'(v) \nabla v$.
Thus, we have
\begin{align*}
&\| \nabla  \phi(v) -\nabla \phi(v_n)  \|_{L^2(U)} = \| \phi'(v) \nabla v - \phi'(v_n) \nabla v_n	\|_{L^2(U)}  \\
&\leq \| \phi'(v) \nabla v - \phi'(v_n) \nabla v \|_{L^2(U)} + \| \phi'(v_n) \nabla v - \phi'(v_n) \nabla v_n \|_{L^2(U)} \\
& \leq \| \phi'(v) \nabla v - \phi'(v_n) \nabla v \|_{L^2(U)}+\|\phi' \|_{L^{\infty}(\mathbb{R})} \| \nabla v- \nabla v_n \|_{L^2(U)}. 
\end{align*}
The first term converges to zero by the Lebesgue dominated convergence theorem and \eqref{pointconver}, and the second term converges to zero by \eqref{h12conver}. Therefore, we have $\phi(v) \in H^{1,2}_0(U)$.
\end{proof}
\centerline{}
The following weak maximum principle originates from N. S. Trudinger \cite{Tma}, and a reformulated version is given in \cite[Chapter 2]{BKRS15} under the assumption that $\mathbf{H} \in L^p(U, \mathbb{R}^d)$ for some $p \in (d, \infty)$. However, the original result in \cite{Tma} allows the critical case $\mathbf{H} \in L^d(U, \mathbb{R}^d)$. For the reader's convenience, we provide the precise statement and a detailed proof of this version below.
\begin{prop}{\bf (Weak maximum principle)} \label{weakmaxim}
Assume {\bf (T1)}. Let $\mathcal{B}: H^{1,2}_0(U) \times H^{1,2}_0(U) \to \mathbb{R}$ denote the bilinear form defined by \eqref{bilinearf}.  Let $u \in H^{1,2}_0(U)$ satisfy
\begin{equation} \label{negativecondiwe}
\mathcal{B}(u, \varphi) \leq 0 \quad \text{ for all $\varphi \in H^{1,2}_0(U)$ with $\varphi \geq 0$ in $U$}.
\end{equation}
Then, $u \leq 0$ in $U$.
\end{prop}
\begin{proof}
Let $\phi \in C^1((-\varepsilon, \infty))$ with $\varepsilon>0$ be such that $\phi \geq 0$ on $[0, \infty)$, $\phi(0) = 0$, and $\phi' \in L^{\infty}((0, \infty))$.
By \cite[Theorem 4.4]{EG15} and Lemma \ref{basicapplem}, we have $u^+ \in H^{1,2}_0(U)$ and $\phi(u^+) \in H^{1,2}_0(U)$, with
\[
\nabla \phi(u^+) = \phi'(u^+) \nabla u^+ = \phi'(u^+) 1_{\{u>0\}} \nabla u \in L^2(U, \mathbb{R}^d).
\]
Substituting $\varphi = \phi(u^+)$ into \eqref{negativecondiwe}, we obtain
\[
\int_U \langle A \nabla u,\, \phi'(u^+) 1_{\{u>0\}} \nabla u \rangle \,dx 
+ \int_U \langle u \mathbf{H},\, \phi'(u^+) 1_{\{u>0\}} \nabla u \rangle \,dx \leq 0.
\]
Since $1_{\{u>0\}} = (1_{\{u>0\}})^2$ and $1_{\{u>0\}} \nabla u = \nabla u^+$, it follows that
\begin{equation} \label{underlyingineq}
\int_U \langle A \nabla u^+,\, \phi'(u^+) \nabla u^+ \rangle \,dx 
\leq  \int_U \langle- u^+ \mathbf{H},\, \phi'(u^+) \nabla u^+ \rangle \,dx.
\end{equation}
Given $\varepsilon > 0$, define $\phi_\varepsilon \in C^1((-\varepsilon, \infty))$ by
\[
\phi_\varepsilon(t) := \frac{t}{t + \varepsilon}, \quad t \in [0, \infty).
\]
Then, clearly $\phi_\varepsilon \geq 0$ on $[0, \infty)$, $\phi_\varepsilon(0) = 0$, and $\phi_\varepsilon' \in L^\infty((0, \infty))$, where
\[
\phi_\varepsilon'(t) = \frac{\varepsilon}{(t + \varepsilon)^2} \quad \text{for all } t \in [0, \infty).
\]
Thus, substituting $\phi = \phi_{\varepsilon}$ into \eqref{underlyingineq} yields
\begin{align} \label{finmaxineq}
\int_U \frac{1}{(u^+ + \varepsilon)^2} \langle A \nabla u^+,\, \nabla u^+ \rangle \,dx 
\leq \int_U \left\langle -u^+ \mathbf{H},\, \frac{1}{(u^+ + \varepsilon)^2} \nabla u^+ \right\rangle \,dx.
\end{align}
For each $\varepsilon > 0$, define $\psi_{\varepsilon} \in C^1((-\varepsilon, \infty))$ by
\[
\psi_{\varepsilon}(t) := \ln \left(1 + \frac{t}{\varepsilon} \right), \quad t \in [0, \infty).
\]
Then, $\psi_{\varepsilon}' \in L^{\infty}((0, \infty))$ satisfies
\[
\psi_{\varepsilon}'(t) = \frac{1}{t + \varepsilon} \quad \text{for all } t \in [0, \infty).
\]
Again, by Lemma~\ref{basicapplem}, $\psi_{\varepsilon}(u^+) \in H^{1,2}_0(U)$, and inequality~\eqref{finmaxineq} implies that
\begin{align*}
\lambda \| \nabla \psi_{\varepsilon}(u^+) \|^2_{L^2(U)}
&\leq \int_U \langle A \nabla \psi_{\varepsilon}(u^+),\, \nabla \psi_{\varepsilon}(u^+) \rangle \,dx \\
&\leq \int_U \left\langle \frac{-u^+}{u^+ + \varepsilon} \mathbf{H},\, \nabla \psi_{\varepsilon}(u^+) \right\rangle \,dx \\
&\leq \| \mathbf{H} \|_{L^2(U)} \| \nabla \psi_{\varepsilon}(u^+) \|_{L^2(U)}.
\end{align*}
Hence,
\[
\| \nabla \psi_{\varepsilon}(u^+) \|_{L^2(U)} \leq \frac{1}{\lambda} \| \mathbf{H} \|_{L^2(U)}.
\]
By Proposition~\ref{mainpoincare}, it follows that
\[
\| \psi_{\varepsilon}(u^+) \|^2_{L^2(U)} \leq \frac{4(d - 1)^2}{\lambda^2 d^2} |U|^{\frac{2}{d}} \| \mathbf{H} \|^2_{L^2(U)}.
\]
Now, suppose there exists a measurable subset $V \subset U$ with $|V| > 0$ such that $u^+(x) > 0$ for all $x \in V$. Then, by Fatou's lemma,
\[
\infty = \int_V \liminf_{\varepsilon \to 0^+} | \psi_{\varepsilon}(u^+) |^2 \,dx
\leq \liminf_{\varepsilon \to 0^+} \int_V | \psi_{\varepsilon}(u^+) |^2 \,dx
\leq \frac{4(d - 1)^2}{\lambda^2 d^2} |U|^{\frac{2}{d}} \| \mathbf{H} \|^2_{L^2(U)}<\infty,
\]
which is a contradiction. Therefore, $u^+ = 0$ a.e. in $U$, as desired.
\end{proof}

\begin{cor} \label{existeanduniqco}
Assume {\bf (T1)}. Let $\mathcal{B}: H^{1,2}_0(U) \times H^{1,2}_0(U) \to \mathbb{R}$ denote the bilinear form defined by \eqref{bilinearf}.
Let $g \in H^{-1,2}(U)$. Then, there exists a unique function $u_g \in H^{1,2}_0(U)$ such that
\begin{equation*} \label{varidenti}
\mathcal{B}(u_g, \varphi)  = \langle g, \varphi \rangle_{H^{-1,2}(U)} \quad \text{ for all $\varphi \in H^{1,2}_0(U)$}.
\end{equation*}
\end{cor}
\begin{proof}
Let $w \in H^{1,2}_0(U)$ satisfy
$$
\mathcal{B}(w, \varphi)=0 \quad \text{ for all $\varphi \in H^{1,2}_0(U)$}.
$$
Then, by Proposition \ref{weakmaxim}, $w \leq 0$ and $-w \leq 0$ in $U$, and hence $w=0$ in $U$. Thus, the assertion follows from Proposition \ref{laxfredexist}.
\end{proof}
\centerline{}
The following lemma provides a standard energy estimate, but due to the assumption that $\mathbf{H} \in L^d(U, \mathbb{R}^d)$, a delicate use of partition of unity and compactness arguments is required.
In contrast, if one assumes $\mathbf{H} \in L^p(U, \mathbb{R}^d)$ for some $p > d$, the estimate could likely be derived more easily using an interpolation inequality. We refer to \cite{Sha} for the derivation of the $H^{1,q}$-estimate under appropriate regularity assumptions on the coefficient matrix $A$ and the domain $U$.
\begin{lem}\label{lempactprelim}
Assume {\bf (T1)}. For each $n \geq 1$, let $A_n = (a_{n, ij})_{1 \leq i,j \leq d}$ be a matrix of measurable functions on $\mathbb{R}^d$, satisfying 
\begin{equation} \label{anuniformel}
\max_{1 \leq i,j \leq d} |a_{n, ij}(x)| \leq M, \quad \langle A_n(x) \xi, \xi \rangle \geq \lambda \| \xi \|^2 \quad \text{for a.e. } x \in \mathbb{R}^d \text{ and all } \xi \in \mathbb{R}^d.
\end{equation}
Assume also that $\lim_{n \rightarrow \infty} a_{n, ij} = a_{ij}$ in $L^2(U)$ for all $1 \leq i, j \leq d$. Let $\eta$ be a standard mollifier on $\mathbb{R}^d$, and for each $n \geq 1$, define $\eta_n \in C_0^{\infty}(B_{1/n})$ given by $\eta_{n}(x):=n^d \eta(nx)$, $x \in \mathbb{R}^d$.
Define
\[
\mathbf{H}_n := \mathbf{H} * \eta_n, \quad n \geq 1,
\]
where $\mathbf{H}$ is the zero extension of $\mathbf{H} \in L^d(U, \mathbb{R}^d)$ to $\mathbb{R}^d$.  Given $g \in H^{-1,2}(U)$, let $u_{n, g} \in H_0^{1,2}(U)$ be the unique function satisfying
\begin{equation} \label{varidenapr}
\int_U \langle A_n \nabla u_{n,g} + u_{n, g}\,\mathbf{H}_n, \nabla \varphi \rangle \, dx  = \langle g, \varphi \rangle_{H^{-1,2}(U)} \quad \text{for every } \varphi \in H^{1,2}_0(U),
\end{equation}
as in Corollary~\ref{existeanduniqco}. Then, there exist constants $c_1, c_2>0$ which only depend on $d$, $\lambda$, $M$, $\mathbf{H}$ and $U$ ($c_1, c_2>0$ are independent of $n$ and $g$)
such that
\begin{equation*}
\| \nabla u_{n, g} \|_{L^2(U)} \leq c_1 \| u_{n,g} \|_{L^2(U)} +c_2 \|g\|_{H^{-1,2}(U)}.
\end{equation*}
\end{lem}
\begin{proof}
First, note that for any open sets $V$, $W$ with $\overline{V} \subset W$, 
\begin{equation} \label{ldestimat}
\| \mathbf{H}_n\|_{L^d(V)} \leq \| \mathbf{H}\|_{L^d(W)}
\end{equation}
(see the proof of \cite[Theorem 7, Appendices]{E10}).
Let $x \in \overline{U}$ and $r_x>0$ be such that
\begin{equation} \label{ldestimat2}
 \frac{2(d-1)}{d-2} 
 \| \mathbf{H}\|_{L^d(B_{2r_x}(x))} \leq \frac{\lambda}{4}.
\end{equation}
Let $\zeta \in C_0^{\infty}(B_{r_x}(x))$. Given $g \in H^{-1,2}(U)$, substituting $\varphi = \zeta^2 u_{n,g} \in H^{1,2}_0(U)$ in \eqref{varidenapr} and using \eqref{imptsobolmud}, we have
\begin{align*}
&\lambda \| \zeta \nabla u_{n,g} \|_{L^2(U)}^2 \leq \int_U \langle A_n \nabla u_{n,g},  \zeta^2 \nabla u_{n,g} \rangle \, dx  \\
&= -\int_{U} \langle A_n \nabla u_{n,g}, 2 u_{n,g}\zeta \nabla \zeta \rangle\, dx
-\int_{U} \langle u_{n,g}\mathbf{H}_n, 2\zeta u_{n,g} \nabla \zeta \rangle \,dx \\
&\quad \qquad - \int_{U} \langle u_{n,g}\mathbf{H}_n,  \zeta^2 \nabla u_{n,g}\rangle \,dx + \langle g, \zeta^2 u_{n,g} \rangle_{H^{-1,2}(U)}  \\
&\leq 2dM \|\zeta \nabla u_{n,g} \|_{L^2(U)} \| u_{n,g} \nabla \zeta\|_{L^2(U)} + 2\| \mathbf{H}_n \|_{L^d(B_{r_x}(x))} \| \zeta u_{n,g} \|_{L^{\frac{2d}{d-2}}(U)} \|  u_{n,g} \nabla \zeta \|_{L^2(U)} \\
& \qquad + \| \mathbf{H}_n\|_{L^d(B_{r_x}(x))}
\| \zeta u_{n,g} \|_{L^{\frac{2d}{d-2}}(U)} \| \zeta \nabla u_{n,g} \|_{L^2(U)} + \|g\|_{H^{-1,2}(U)} \left( \|\zeta^2 \nabla u_{n,g} \|_{L^2(U)}		
+ \| 2\zeta u_{n,g} \nabla \zeta \|_{L^2(U)} \right) \\
&\leq   2dM \|\zeta \nabla u_{n,g} \|_{L^2(U)} \| u_{n,g} \nabla \zeta\|_{L^2(U)} + \frac{4(d-1)}{d-2}\| \mathbf{H}_n \|_{L^d(B_{r_x}(x))} \| \nabla (\zeta u_{n,g}) \|_{L^2(U)} \|  u_{n,g} \nabla \zeta \|_{L^2(U)} \\
&   \quad \qquad + \frac{2(d-1)}{d-2}
\| \mathbf{H}_n\|_{L^d(B_{r_x}(x))}
\| \nabla (\zeta u_{n,g}) \|_{L^{2}(U)} \| \zeta \nabla u_{n,g} \|_{L^2(U)} \\
& \quad \qquad  + \|g\|_{H^{-1,2}(U)} \|\zeta \|_{L^{\infty}(U)}  \|\zeta \nabla u_{n,g} \|_{L^2(U)}		
+ \|g\|_{H^{-1,2}(U)} \|2 \zeta\|_{L^{\infty}(U)}
\|  u_{n,g} \nabla \zeta \|_{L^2(U)}.
\end{align*}
Using Young's inequality, we obtain that
\begin{align}
&\frac{\lambda}{2} \| \zeta \nabla u_{n,g} \|_{L^2(U)}^2 \nonumber \\
&\leq \left( \frac{8d^2M^2}{\lambda}  + \frac{40(d-1)^2}{\lambda(d-2)^2} \| \mathbf{H}_n\|^2_{L^d(B_{r_x}(x))}+  \frac{4(d-1)}{d-2} \|\mathbf{H}_n\|_{L^d(B_{r_x}(x))}
 +2 \| \zeta \|^2_{L^{\infty}(U)}  \right)
\| u_{n,g} \nabla \zeta  \|^2_{L^2(U)}  \nonumber \\
& \; \;+ \left( \frac{2}{\lambda} \|\zeta\|_{L^{\infty}(U)}^2 + \frac12	\right) \| g \|^2_{H^{-1,2}(U)} + \frac{2(d-1)}{d-2} \| \mathbf{H}_n \|_{L^d(B_{r_x}(x))} \| \zeta \nabla u_{n,g}  \|_{L^2(U)}^2.  \label{underl2estimin}
\end{align}
Applying \eqref{ldestimat} and \eqref{ldestimat2} to \eqref{underl2estimin},
\begin{align*}
&\| \zeta \nabla u_{n,g} \|_{L^2(U)}^2 \\
&\leq  \frac{4}{\lambda}  \left( \frac{8d^2M^2}{\lambda}  + \frac{40(d-1)^2}{\lambda(d-2)^2} \| \mathbf{H}\|^2_{L^d(U)}+  \frac{4(d-1)}{d-2} \|\mathbf{H}\|_{L^d(U)}
 +2 \| \zeta \|^2_{L^{\infty}(U)}  \right) \| \nabla \zeta \|^2_{L^{\infty}(U)}  \| u_{n,g} \|^2_{L^2(U)}  \\
& \qquad \qquad  +\frac{4}{\lambda} \left( \frac{2}{\lambda} + \frac12	\right) \| g \|^2_{H^{-1,2}(U)}.
\end{align*}
Since $\overline{U}$ is compact and 
$\{ B_{r_x}(x) :  x \in \overline{U}\}$
is an open cover of $\overline{U}$, there exists $x_1, \ldots, x_N \in \overline{U}$ such that
$$
\overline{U} \subset \bigcup_{i=1}^N B_{r_{x_i}}(x_i).
$$
Let $(\zeta_i)_{i=1}^N$ be the smooth partition of unity with $\text{supp}(\zeta_i) \subset B_{r_{x_i}}(x_i)$ such that
$$
\sum_{i=1}^N \zeta_i =1 \quad \text{ on $\overline{U}$}.
$$
 Therefore,
 \begin{align*}
\| \nabla u_{n,g} \|_{L^2(U)} &=  \left \| \sum_{i=1}^N\zeta_i \nabla u_{n,g} \right \|_{L^2(U)}  \leq \sum_{i=1}^N \left\| \zeta_i \nabla u_{n,g} \right \|_{L^2(U)} \\
&\leq c_1 \| u_{n,g} \|_{L^2(U)} +c_2 \|g\|_{H^{-1,2}(U)},
 \end{align*}
 where 
 \begin{align*}
c_1&=\sum_{i=1}^N\frac{2}{\sqrt{\lambda}}\left( \frac{8d^2M^2}{\lambda}  + \frac{40(d-1)^2}{\lambda(d-2)^2} \| \mathbf{H}\|^2_{L^d(U)}+  \frac{4(d-1)}{d-2} \|\mathbf{H}\|_{L^d(U)}
 +2 \| \zeta_i \|^2_{L^{\infty}(U)}  \right)^{\frac12} \| \nabla \zeta_i \|_{L^{\infty}(U)} 
 \end{align*}
 and
$$ 
 c_2=\frac{2N}{\sqrt{\lambda}} \left( \frac{2}{\lambda} + \frac12	\right)^{\frac12}.
$$
\end{proof}
\centerline{}
The following lemma is inspired by the compactness arguments in \cite[Section 6.2, Theorem 6]{E10}, and its key feature is that the constant $C > 0$ remains independent of both the index $n$ and the external data $g \in H^{-1,2}(U)$, even though the coefficients are given as a sequence rather than a single function. Different from \cite[Lemma 3.3]{LT21}, the main feature here is that uniform estimates are obtained for the mollifications of $\mathbf{H}$, assuming $\mathbf{H} \in L^d(U, \mathbb{R}^d)$.
\begin{lem}\label{lempactness}
Assume {\bf (T1)}. For each $n \geq 1$, let $A_n = (a_{n, ij})_{1 \leq i,j \leq d}$ be a matrix of measurable functions on $\mathbb{R}^d$ satisfying 
\eqref{anuniformel}. Assume also that $\lim_{n \rightarrow \infty} a_{n, ij} = a_{ij}$ in $L^2(U)$ for all $1 \leq i, j \leq d$. Let $\eta$ be a standard mollifier on $\mathbb{R}^d$, and for each $n \geq 1$, define $\eta_n \in C_0^{\infty}(B_{1/n})$ given by $\eta_{n}(x):=n^d \eta(nx)$, $x \in \mathbb{R}^d$.  Define
\[
\mathbf{H}_n := \mathbf{H} * \eta_n,
\]
where $\mathbf{H}$ is the zero extension of $\mathbf{H} \in L^d(U, \mathbb{R}^d)$ to $\mathbb{R}^d$.  Given $g \in H^{-1,2}(U)$, let $u_{n, g} \in H_0^{1,2}(U)$ be the unique function satisfying
\[
\int_U \langle A_n \nabla u_{n,g} + u_{n, g}\,\mathbf{H}_n, \nabla \varphi \rangle \, dx  = \langle g, \varphi \rangle_{H^{-1,2}(U)} \quad \text{for every } \varphi \in H^{1,2}_0(U),
\]
as in Corollary~\ref{existeanduniqco}. Then, the following statements hold:
\begin{itemize}
\item[(i)]
There exists a constant $C>0$ independent of $n \geq 1$ and $g \in H^{-1,2}(U)$ such that
\begin{equation} \label{uniformestimacom}
\| u_{n, g} \|_{L^{2}(U)} \leq C \| g \|_{H^{-1,2}(U)} \quad \text{for all } n \geq 1 \text{ and } g \in H^{-1,2}(U).
\end{equation}
Moreover, 
\begin{equation} \label{appenergestim}
\| \nabla u_{n, g} \|_{L^2(U)} \leq  (c_1 C+c_2) \|g\|_{H^{-1,2}(U)},
\end{equation}
where $c_1, c_2>0$ are constants as in Lemma \ref{lempactprelim}.
\item[(ii)]
Given $g \in H^{-1,2}(U)$, let $u_g \in H^{1,2}_0(U)$ be the unique function satisfying
\[
\int_U \langle A \nabla u_{g} + u_{g}\,\mathbf{H}, \nabla \varphi \rangle \, dx  = \langle g, \varphi \rangle_{H^{-1,2}(U)} \quad \text{for every } \varphi \in H^{1,2}_0(U),
\]
as in Corollary~\ref{existeanduniqco}.
Then, there exists a subsequence of $(u_{n,g})_{n \geq 1}$, say again $(u_{n,g})_{n \geq 1}$ such that
\begin{equation} \label{weakconl2con}
\lim_{n \rightarrow \infty} u_{n,g} = u_g \quad \text{ weakly in $H^{1,2}_0(U)$} \quad \text{ and } \quad \lim_{n \rightarrow \infty} u_{n,g} = u_g \quad \text{ in $L^{2}(U)$}.
\end{equation}
In particular, 
\begin{equation*}
\| u_{g} \|_{L^{2}(U)} \leq C \| g \|_{H^{-1,2}(U)} \quad \text{for all } n \geq 1 \text{ and } g \in H^{-1,2}(U).
\end{equation*}
and
$$
\| \nabla u_{g} \|_{L^2(U)} \leq  (c_1 C+c_2) \|g\|_{H^{-1,2}(U)},
$$
where $C>0$ is the constant as in (i) and  $c_1, c_2>0$ are constants as in Lemma \ref{lempactprelim}.
\end{itemize}
\end{lem}
\begin{proof}
(i)
Suppose, by contradiction, that \eqref{uniformestimacom} does not hold. Then, for each $k \in \mathbb{N}$, there exist $\tilde{g}_k \in H^{-1,2}(U)$ and $n_k \in \mathbb{N}$ such that
\[
\| u_{n_k, \tilde{g}_k} \|_{L^{2}(U)} > k \| \tilde{g}_k \|_{H^{-1,2}(U)}.
\]
Define 
\[
g_k := \frac{\tilde{g}_k}{\| u_{n_k, \tilde{g}_k} \|_{L^{2}(U)}} \in H^{-1,2}(U).
\]
By Corollary~\ref{existeanduniqco}, it follows that
\[
u_{n_k, g_k} = \frac{u_{n_k, \tilde{g}_k}}{\| u_{n_k, \tilde{g}_k} \|_{L^{2}(U)}}.
\]
Therefore,
\begin{equation} \label{unkgk1}
\| u_{n_k, g_k} \|_{L^{2}(U)} = 1
\end{equation}
and
\begin{equation} \label{gkestimate}
\| g_k \|_{H^{-1,2}(U)} < \frac{1}{k}.
\end{equation}
Meanwhile, we have
\[
\int_U \langle A_{n_k} \nabla u_{n_k,g_k} + u_{n_k, g_k}\,\mathbf{H}_{n_k}, \nabla \varphi \rangle\, dx = \langle g_k, \varphi \rangle_{H^{-1,2}(U)} \quad \text{for every } \varphi \in H^{1,2}_0(U).
\]
Using Lemma \ref{lempactprelim}, \eqref{unkgk1} and \eqref{gkestimate},  it follows that
\begin{align}
\| \nabla u_{n_k,g_k} \|_{L^2(U)} &\leq  c_1 \| u_{n,g} \|_{L^2(U)} +c_2 \|g\|_{H^{-1,2}(U)}  \nonumber \\
&\leq c_1 +c_2, \label{compestima}
\end{align}
where $c_1, c_2>0$ are constants which only depend on $d$, $\lambda$, $M$, $\mathbf{H}$ and $U$ ($c_1, c_2$ are independent of $n$ and $g$).\\
{\underline{\sf Case 1})} Suppose that the set $\{ n_k : k \geq 1 \}$ is bounded. Then, there exists $N \in \mathbb{N}$ and a subsequence $(k_j)_{j \geq 1} \subset (k)_{k \geq 1}$ such that $n_{k_j} = N$ for all $j \geq 1$. In this case, from \eqref{compestima}, we deduce that
\[
\| \nabla u_{N, g_{k_j}} \|_{L^2(U)} \leq c_1+c_2 \quad \text{ for all $j \geq 1$.}
\]
Moreover,
\begin{equation} \label{divformgk}
\int_U \langle A_N \nabla u_{N, g_{k_j}} + u_{N, g_{k_j}}\,\mathbf{H}_N, \nabla \varphi \rangle\, dx = \langle g_{k_j}, \varphi \rangle_{H^{-1,2}(U)} \quad \text{for all } \varphi \in H^{1,2}_0(U) \; \text{ and $j \geq 1$}.
\end{equation}
By the weak compactness of bounded sets in $H_0^{1,2}(U)$ and the Rellich-Kondrachov compactness theorem, there exists a subsequence of $(u_{N, g_{k_j}})_{j \geq 1}$, which we denote again by $(u_{N, g_{k_j}})_{j \geq 1}$, and a function $u \in H_0^{1,2}(U)$ such that
\begin{equation} \label{weakcomus}
\lim_{j \rightarrow \infty} u_{N, g_{k_j}} =u \quad \text{weakly in } H_0^{1,2}(U), \qquad  \lim_{j \rightarrow \infty} u_{N, g_{k_j}} =u \quad \text{ in } L^{2}(U).
\end{equation}
Passing to the limit in \eqref{divformgk} along this subsequence and using the fact that $g_{k_j} \to 0$ in $H^{-1,2}(U)$ as $j \to \infty$ (see \eqref{gkestimate}), we obtain
\begin{equation*} \label{divformgklim}
\int_U \langle A_{N} \nabla u + u\,\mathbf{H}_{N}, \nabla \varphi \rangle \, dx = 0 \quad \text{for all } \varphi \in H^{1,2}_0(U).
\end{equation*}
By the uniqueness in Corollary~\ref{existeanduniqco}, it follows that $u = 0$ in $U$. On the other hand, by \eqref{unkgk1} and \eqref{weakcomus}, 
\[
1 =\lim_{j \rightarrow \infty} \|  u_{N, g_{k_j}} \|_{L^2(U)} = \| u \|_{L^2(U)}=0,
\]
which is a contradiction. \\
{\underline{\sf Case 2})} Suppose now that the set $\{ n_k : k \geq 1 \}$ is unbounded. Then, there exists a subsequence $(k_j)_{j \geq 1} \subset (k)_{k \geq 1}$ such that
\[
\lim_{j \to \infty} n_{k_j} = \infty
\]
and
\begin{equation} \label{divformgkcase2}
\int_U \langle A_{n_{k_j}} \nabla u_{n_{k_j}, g_{k_j}} + u_{n_{k_j}, g_{k_j}}\,\mathbf{H}_{n_{k_j}}, \nabla \varphi \rangle\, dx = \langle g_{k_j}, \varphi \rangle_{H^{-1,2}(U)} \quad \text{for all } \varphi \in H^{1,2}_0(U).
\end{equation}
By \eqref{compestima}, we have
\[
\| \nabla u_{n_{k_j}, g_{k_j}} \|_{L^2(U)} \leq c_1+c_2.
\]
Consequently, there exists a subsequence of $(u_{n_{k_j}, g_{k_j}})_{j \geq 1}$, say again $(u_{n_{k_j}, g_{k_j}})_{j \geq 1}$, and a function $u \in H_0^{1,2}(U)$ such that
\begin{equation} \label{weakcomus2}
\lim_{j \rightarrow \infty}u_{n_{k_j}, g_{k_j}} = u \quad \text{weakly in } H_0^{1,2}(U), \qquad \lim_{j \rightarrow \infty}u_{n_{k_j}, g_{k_j}} = u \quad \text{ in } L^{2}(U).
\end{equation}
Using \eqref{weakcomus2}, we now pass to the limit in the weak formulation in \eqref{divformgkcase2}, and hence we get
\[
\int_U \langle A \nabla u + u\,\mathbf{H}, \nabla \varphi \rangle\, dx = 0 \quad \text{for all } \varphi \in H_0^{1,2}(U).
\]
By the uniqueness in Corollary~\ref{existeanduniqco},  it follows that $u=0$ in $U$.
On the other hand, by \eqref{unkgk1} and \eqref{weakcomus2},
\[
1 = \lim_{j\to\infty} \| u_{n_{k_j}, g_{k_j}} \|_{L^{2}(U)} = \| u \|_{L^2(U)}=0.
\]
Since both cases result in a contradiction, the initial assumption must be false. Thus, \eqref{uniformestimacom} does hold. Therefore, \eqref{appenergestim} directly follows from Lemma \ref{lempactprelim}. \\ \\
(ii)  By the weak compactness of bounded subsets in $H_0^{1,2}(U)$ and the Rellich-Kondrachov compactness theorem applied to \eqref{appenergestim}, there exists a subsequence of $(u_{n,g})_{n \geq 1}$, which we still denote by $(u_{n,g})_{n \geq 1}$, such that \eqref{weakconl2con} holds. The rest follows from the lower semi-continuity of the norm used in the estimates \eqref{uniformestimacom} and \eqref{appenergestim}.
\end{proof}
\begin{rem}
Assume {\bf (T1)}, where $d \geq 3$ is replaced by $d = 2$, and suppose that $\mathbf{H} \in L^p(U, \mathbb{R}^2)$ for some $p \in (2, \infty)$. 
In analogy with the proofs of Propositions \ref{bilineaprop}, \ref{laxfredexist}, \ref{weakmaxim} and Corollary \ref{existeanduniqco}, we obtain that for each $g \in H^{-1,2}(U)$, there exists a unique function $u_g \in H^{1,2}_0(U)$ satisfying
\[
\int_U \langle A \nabla u_{g} + u_{g}\,\mathbf{H}, \nabla \varphi \rangle \, dx  = \langle g, \varphi \rangle_{H^{-1,2}(U)} \quad \text{for all } \varphi \in H^{1,2}_0(U).
\]
For each $n \geq 1$, let $A_n = (a_{n, ij})_{1 \leq i,j \leq d}$ be a matrix-valued function satisfying \eqref{anuniformel}, and assume that $\lim_{n \rightarrow \infty} a_{n, ij} = a_{ij}$ in $L^2(U)$ for all $1 \leq i, j \leq d$. Let $\mathbf{H}_n \in L^p(U, \mathbb{R}^2)$ be the mollification of the zero extension of $\mathbf{H}$ to $\mathbb{R}^d$ as in Lemma \ref{lempactness}. For each $n \geq 1$, let $u_{n,g} \in H^{1,2}_0(U)$ be the unique function satisfying
\[
\int_U \langle A_n \nabla u_{n,g} + u_{n, g}\,\mathbf{H}_n, \nabla \varphi \rangle \, dx  = \langle g, \varphi \rangle_{H^{-1,2}(U)} \quad \text{for all } \varphi \in H^{1,2}_0(U).
\]
Then, by a similar argument to that in the proof of Lemma \ref{lempactness}, we obtain the results in Lemma \ref{lempactness}.
\end{rem}
\centerline{}
The following lemma is a generalization of \cite[Lemma 3.4]{LT21}, and in particular, it remains valid even under the assumption $\mathbf{H} \in L^d(U, \mathbb{R}^d)$.
Its proof requires a highly delicate approximation argument, in which Lemma~\ref{lempactness} plays a central role.
\begin{lem} \label{mainkeyprese}
Assume {\bf (T1)}. Assume that $u \in H^{1,2}(U)$ satisfies
\[
\int_U \langle A \nabla u + u\mathbf{H}, \nabla \varphi \rangle dx \leq 0 \quad \text{ for all } \varphi \in C_0^\infty(U),\ \varphi \geq 0.
\]
Then, we have
\[
\int_U \langle A \nabla u^+ + u^+ \mathbf{H}, \nabla \varphi \rangle dx \leq 0 \quad \text{ for all } \varphi \in C_0^\infty(U),\ \varphi \geq 0.
\]
\end{lem}
\begin{proof}
Let $V$ be an arbitrary open subset of $U$ with $\overline{V} \subset U$. To show the assertion, it is enough to show~that
\begin{equation*} 
\int_U \langle A \nabla u^+ + u^+ \mathbf{H}, \nabla \varphi \rangle dx \leq 0 \quad \text{for all } \varphi \in C_0^\infty(V),\ \varphi \geq 0.
\end{equation*}
Let $W$ be an open set with a smooth boundary such that
\[
\overline{V} \subset W \subset \overline{W} \subset U.
\]
Let $B$ be an open ball such that $\overline{U} \subset B$. By \cite[Theorem 4.7]{EG15}, $u \in H^{1,2}(W)$ can be extended to a function $\hat{u} \in H_0^{1,2}(B)$. Moreover, by \cite[Theorem 4.4(iii)]{EG15}, we have $\hat{u}^+ \in H_0^{1,2}(B)$ with
\[
\nabla \hat{u}^+ = \begin{cases}
\nabla \hat{u} & \text{a.e. on } \{ \hat{u} > 0 \}, \\
0 & \text{a.e. on } \{ \hat{u} \leq 0 \}.
\end{cases}
\]
Extend $\mathbf{H} \in L^d(U, \mathbb{R}^d)$ to $\mathbb{R}^d$ by zero extension. Define 
$$
\mathbf{F} := A \nabla \hat{u} + \hat{u}\mathbf{H} \in L^2(\mathbb{R}^d, \mathbb{R}^d). 
$$
Let $\eta$ be a standard mollifier on $\mathbb{R}^d$, and for each $n \geq 1$, define $\eta_n \in C_0^{\infty}(B_{1/n})$ by $\eta_{n}(x):=n^d \eta(nx)$, $x \in \mathbb{R}^d$.
For each $n \in \mathbb{N}$ and $1 \leq i,j \leq d$, define
\[
a_{n, ij} := a_{ij} * \eta_n, \quad A_n := (a_{n, ij})_{1 \leq i,j \leq d}, \quad \mathbf{H}_n := \mathbf{H} * \eta_n, \quad \mathbf{F}_n := \mathbf{F} * \eta_n \quad \text{on } \mathbb{R}^d.
\]
Then, $a_{n, ij} \in C^\infty(\mathbb{R}^d)$, and $\mathbf{H}_n$, $\mathbf{F}_n \in C^\infty(\mathbb{R}^d, \mathbb{R}^d)$ for all $n \geq 1$ and $1 \leq i,j \leq d$, and it holds that
\[
\lim_{n \rightarrow \infty}a_{n, ij} =a_{ij} \;\text{in } L^2(B, \mathbb{R}^d), \quad \lim_{n \rightarrow \infty} \mathbf{H}_n = \mathbf{H} \;\text{ in } L^d(B, \mathbb{R}^d), \quad \lim_{n \rightarrow \infty} \mathbf{F}_n = \mathbf{F} \;\text{ in } L^2(B, \mathbb{R}^d).
\]
Furthermore, \eqref{anuniformel} holds. Choose $\delta > 0$ such that $\overline{B}_\delta(x) \subset W$ for all $x \in \overline{V}$. Pick $N \in \mathbb{N}$ with $\frac{1}{N} < \delta$. Then, for any $n \geq N$ and $\varphi \in C_0^\infty(V)$ with $\varphi \geq 0$, we have $\varphi* \eta_n \in C_0^{\infty}(W)$, $\varphi* \eta_n  \geq 0$, and
\begin{align}
&\int_U \langle \mathbf{F}_n, \nabla \varphi \rangle dx = \int_{\mathbb{R}^d} \langle \mathbf{F}_n, \nabla \varphi \rangle dx = \int_{\mathbb{R}^d} \langle \mathbf{F}, \nabla (\varphi *\eta_n) \rangle dx  \nonumber \\
&= \int_{\mathbb{R}^d} \langle A \nabla \hat{u} + \hat{u}\mathbf{H}, \nabla (\varphi * \eta_n) \rangle dx =\int_{U} \langle A \nabla u + u \mathbf{H}, \nabla (\varphi * \eta_n) \rangle dx \leq 0.  \label{fndivnegati}
\end{align}
According to Corollary \ref{existeanduniqco}, there exists a unique function $u_n \in H_0^{1,2}(B)$ such that
\begin{equation} \label{fndividtest}
\int_B \langle A_n \nabla u_n + u_n \mathbf{H}_n, \nabla \tilde{\varphi} \rangle dx = \int_B \langle \mathbf{F}_n, \nabla \tilde{\varphi} \rangle dx \quad \text{for all } \tilde{\varphi} \in C_0^\infty(B). 
\end{equation}
By Lemma \ref{lempactness}(i), we obtain
\[
\| u_n \|_{H_0^{1,2}(B)} \leq (c_1C + c_2) \| \mathbf{F}_n \|_{L^2(B, \mathbb{R}^d)} \leq (c_1C + c_2)  \| \mathbf{F} \|_{L^2(B, \mathbb{R}^d)},
\]
where $c_1, c_2 > 0$ are constants as in Lemma \ref{lempactprelim} and $C > 0$ is the constant as in Lemma \ref{lempactness}(i).
By the weak compactness of bounded subsets in $H_0^{1,2}(B)$ and using \cite[Theorem 4.4(iii)]{EG15}, there exist $\tilde{u} \in H_0^{1,2}(B)$ and a subsequence (still denoted by $u_n$) such that
\begin{equation} \label{weaklimitplu}
\lim_{n \rightarrow \infty} u_n = \tilde{u} \quad \text{and} \quad \lim_{n \rightarrow \infty} u_n^+ = \tilde{u}^+ \quad \text{weakly in } H_0^{1,2}(B).
\end{equation}
Hence letting $n \rightarrow \infty$, we get
\[
\int_B \langle A \nabla \tilde{u} + \tilde{u} \mathbf{H}, \nabla \tilde{\varphi} \rangle dx = \int_B \langle \mathbf{F}, \nabla \tilde{\varphi} \rangle dx = \int_B \langle A \nabla \hat{u} + \hat{u} \mathbf{H}, \nabla \tilde{\varphi} \rangle dx \quad \text{for all } \tilde{\varphi} \in C_0^\infty(B).
\]
By the uniqueness in Proposition \ref{weakmaxim}, we conclude that $\tilde{u} = \hat{u}$ in $H^{1,2}_0(B)$. Thus, by \eqref{weaklimitplu}, we have
\begin{equation} \label{weaklimitplu2}
\lim_{n \rightarrow \infty} u_n = \hat{u} \quad \text{and} \quad \lim_{n \rightarrow \infty} u_n^+ = \hat{u}^+ \quad \text{weakly in } H_0^{1,2}(B).
\end{equation}
Define the operator
\[
\mathcal{L}_n u_n := \sum_{i,j=1}^d a_{n, ij} \partial_i \partial_j u_n + \langle \mathbf{H}_n + \operatorname{div} A_n, \nabla u_n \rangle + (\operatorname{div} \mathbf{H}_n) u_n.
\]
Then from \eqref{fndivnegati} and \eqref{fndividtest}, we deduce that for all $n \geq N$ and $\varphi \in C_0^\infty(V)$ with $\varphi \geq 0$,
\[
- \int_V \mathcal{L}_n u_n \cdot \varphi dx = \int_V \langle A_n \nabla u_n + u_n \mathbf{H}_n, \nabla \varphi \rangle dx = \int_{U} \langle \mathbf{F}_n, \nabla \varphi \rangle dx \leq 0,
\]
which implies 
\begin{equation} \label{ptwiselvpr}
\mathcal{L}_n u_n \geq 0 \quad \text{in } V \quad \text{for all } n \geq N.
\end{equation}
Let $\phi$ be a standard mollifier on $\mathbb{R}$, and for each $n \geq 1$, define $\phi_n \in C_0^{\infty}(-1/n, 1/n)$ by $\phi_n(t) := n \phi(nt)$ for $t \in \mathbb{R}$. 
For each $\varepsilon > 0$, define
\[
f_\varepsilon(z) := 
\begin{cases}
\sqrt{z^2 + \varepsilon^2} - \varepsilon & \text{if } z \geq 0, \\
0 & \text{if } z < 0.
\end{cases}
\]
Then, $f_\varepsilon \in C^1(\mathbb{R})$ and its derivative $f_\varepsilon'$ belongs to $H^{1,\infty}(\mathbb{R}) \cap C(\mathbb{R})$. In particular, we have
\[
f_\varepsilon'(z) = 
\begin{cases}
\frac{z}{\sqrt{z^2 + \varepsilon^2}} & \text{if } z \geq 0, \\
0 & \text{if } z < 0,
\end{cases}
\quad \text{and} \quad
f_\varepsilon''(z) = 
\begin{cases}
\frac{\varepsilon^2}{(z^2 + \varepsilon^2)^{3/2}} & \text{if } z > 0, \\
0 & \text{if } z < 0.
\end{cases}
\]
Observe that 
\begin{equation} \label{varepconzero}
\lim_{\varepsilon \rightarrow 0+} f_\varepsilon(z) = z^+, \quad \text{and} \quad \lim_{\varepsilon \rightarrow 0+} f_\varepsilon'(z) = 1_{(0,\infty)}(z) \quad \text{for all } z \in \mathbb{R}.
\end{equation}
Let $f_{\varepsilon, k} := f_\varepsilon * \phi_k$. Then $f'_{\varepsilon, k} \geq 0$ and $f''_{\varepsilon, k} \geq 0$ on $\mathbb{R}$. Moreover,
\begin{equation} \label{uniformconver}
\lim_{k \rightarrow \infty} f_{\varepsilon, k}(u_n) = f_\varepsilon(u_n), \quad \text{and} \quad \lim_{k \rightarrow \infty} f'_{\varepsilon, k}(u_n) = f_\varepsilon'(u_n) \quad \text{uniformly on } U.
\end{equation}
Thus, for any $\varphi \in C_0^\infty(V)$ with $\varphi \geq 0$, it follows from \eqref{uniformconver} and \eqref{ptwiselvpr} that
\begin{align*}
&\int_U \langle A_n \nabla f_\varepsilon(u_n) + f_\varepsilon(u_n) \mathbf{H}_n, \nabla \varphi \rangle dx 
= \lim_{k \to \infty} \int_U \langle A_n \nabla f_{\varepsilon, k}(u_n) + f_{\varepsilon, k}(u_n) \mathbf{H}_n, \nabla \varphi \rangle dx \\
&= \lim_{k \to \infty} \Bigg( - \int_U f'_{\varepsilon, k}(u_n) \mathcal{L}_n u_n \varphi \, dx 
- \int_U f''_{\varepsilon, k}(u_n) \langle A_n \nabla u_n, \nabla u_n \rangle \varphi \, dx \\
&\qquad \qquad \qquad - \int_U \operatorname{div} \mathbf{H}_n \big(f_{\varepsilon, k}(u_n) - u_n f'_{\varepsilon, k}(u_n)\big) \varphi \, dx \Bigg) \\
&\leq - \int_U \operatorname{div} \mathbf{H}_n \big(f_{\varepsilon}(u_n) - u_n f'_{\varepsilon}(u_n)\big) \varphi \, dx.
\end{align*}
Since the right-hand side tends to zero as $\varepsilon \to 0^+$ by \eqref{varepconzero}, we conclude from \cite[Theorem 4.4(iii)]{EG15} and \eqref{varepconzero} that for all $n \geq N$,
\[
\int_U \langle A_n \nabla u_n^+  +u_n^+\mathbf{H}_n, \nabla \varphi \rangle dx \leq 0 \quad \text{ for all } \varphi \in C_0^\infty(V),\ \varphi \geq 0.
\]
Finally, by taking the weak limit of $u_n^+$ in \eqref{weaklimitplu2} as $n \rightarrow \infty$, we obtain
\[
\int_U \langle A \nabla u^+ + u^+\mathbf{H}, \nabla \varphi \rangle dx \leq 0 \quad \text{ for all } \varphi \in C_0^\infty(V),\ \varphi \geq 0,
\]
which completes the proof.
\end{proof}

\section{Proof of main result} \label{sec4}
 
The following theorem is a key result of this paper, which corresponds to \cite[Theorem 3.1]{L25jm} but weakens the assumption on $\mathbf{H}$ from $L^p(U, \mathbb{R}^d)$ to $L^d(U, \mathbb{R}^d)$ by taking advantage of the additional structure on its divergence. This additional structure allows us to apply H\"{o}lder regularity and the Harnack inequality. The idea of the proof originates from \cite[Theorem 1]{BRS12} (cf.~\cite[Chapter 2]{BKRS15}), where the coefficient matrix $A$ is assumed to lie in $VMO$.

\begin{theo} \label{existrho}
Assume that {\bf (T)} holds. Then the following statements hold:
\begin{itemize}
\item[(i)]
Let $x_1 \in U$. Then, there exists $\rho \in H^{1,2}(B_{r}(x_0)) \cap C(B_{r}(x_0))$ with $\rho(x)>0$ for all $x \in B_{r}(x_0)$ and $\rho(x_1)=1$ such that
\begin{equation} \label{infinva}
\int_{B_{r}(x_0)} \langle A^T \nabla \rho  + \rho \mathbf{H}, \nabla \varphi \rangle dx = 0, \quad \text{ for all $\varphi \in C_0^{\infty}(B_{r}(x_0))$}.
\end{equation}
\item[(ii)]
Let $\rho$ be as in Theorem \ref{existrho}(i). Then, there exists a constant $\tilde{K}_1 \geq 1$ which only depends on $d$, $\lambda$, $M$, $B_r(x_0)$, $p$, $\tilde{q}$ and $\mathbf{H}$ such that 
$$
1 \leq \max_{\overline{U}} \rho  \leq \tilde{K}_1 \min_{\overline{U}} \rho \leq \tilde{K}_1.
$$
\end{itemize}
\end{theo}
\begin{proof}
(i) By Corollary \ref{existeanduniqco}, there exists $v \in H^{1,2}_0(B_r(x_0))$ such that
\begin{equation} \label{infinvahbilin}
\int_{B_{r}(x_0)} \langle A^T \nabla v  + v \mathbf{H}, \nabla \varphi \rangle dx = -\int_{B_r(x_0)} \langle \mathbf{H}, \nabla \varphi \rangle\,dx \quad \text{ for all $\varphi \in H_0^{1,2}(B_{r}(x_0))$}.
\end{equation}
Let $w = v+1 \in H^{1,2}(B_r(x_0))$. Let $\mathcal{T}:H^{1,2}(B_r(x_0)) \rightarrow L^2(\partial B_r(x_0))$ be the trace operator as in \cite[Theorem 4.6]{EG15}. Then, 
\begin{equation} \label{tracewone}
\mathcal{T}(w)=\mathcal{T}(v)+1=1 \quad \text{  in $L^2(\partial B_r(x_0))$.} 
\end{equation}
Observe that from \eqref{infinvahbilin}
\begin{equation} \label{infinvahbilin2}
\int_{B_{r}(x_0)} \langle A^T \nabla w  + w \mathbf{H}, \nabla \varphi \rangle dx = 0 \quad \text{ for all $\varphi \in H_0^{1,2}(B_{r}(x_0))$}.
\end{equation}
Meanwhile, $-w = -v-1 \leq -v$ in $B_r(x_0)$, and hence $0 \leq (-w)^+ \leq (-v)^+$ in $B_r(x_0)$. Since $(-v)^+ \in H^{1,2}_0(B_r(x_0))$, it follows by \cite[Proposition A.9]{L25jm} that $(-w)^+ \in H^{1,2}_0(B_r(x_0))$. Therefore, applying Lemma \ref{mainkeyprese} to \eqref{infinvahbilin2} where $w$ is replaced by $-w$, we have
\begin{equation*} \label{infinvahbilin3}
\int_{B_{r}(x_0)} \langle A^T \nabla (-w)^+  + (-w)^+ \mathbf{H}, \nabla \varphi \rangle dx \leq 0 \quad \text{ for all $\varphi \in H_0^{1,2}(B_{r}(x_0))$ with $\varphi \geq 0$}.
\end{equation*}
By Proposition \ref{weakmaxim},  $(-w)^+ \leq 0$ in $B_r(x_0)$, which implies
\begin{equation} \label{wpositivcon}
w \geq 0 \quad \text{ in $B_r(x_0)$}.
\end{equation}
Let $w_n \in C_0^{\infty}(B_r(x_0))$ be such that $\lim_{n \rightarrow \infty} w_n =w$ in $H^{1,2}_0(B_r(x_0))$. For each $\varphi \in C_0^{\infty}(B_r(x_0))$, we have
\begin{align*}
&\int_{B_r(x_0)} \langle w\mathbf{H}, \nabla \varphi \rangle\,dx =\int_{B_r(x_0)}\langle w\mathbf{H}_1, \nabla \varphi \rangle\,dx +\int_{B_r(x_0)}\langle w\mathbf{H}_2, \nabla \varphi \rangle\,dx \\
&=\int_{B_r(x_0)}\langle w\mathbf{H}_1, \nabla \varphi \rangle\,dx +\lim_{n \rightarrow \infty} \left(\int_{B_r(x_0)} \langle \mathbf{H}_2, \nabla (w_n\varphi) \rangle\,dx - \int_{B_r(x_0)} \langle \mathbf{H}_2, \varphi \nabla w_n \rangle\,dx \right) \\
&=\int_{B_r(x_0)}\langle w\mathbf{H}_1, \nabla \varphi \rangle\,dx +\lim_{n \rightarrow \infty} \left(-\int_{B_r(x_0)} \tilde{h} w_n\varphi \,dx - \int_{B_r(x_0)} \langle \mathbf{H}_2, \varphi \nabla w_n \rangle\,dx \right) \\
&=\int_{B_r(x_0)}\langle w\mathbf{H}_1, \nabla \varphi \rangle\,dx- \int_{B_r(x_0)} \langle \mathbf{H}_2, \nabla w \rangle \varphi \,dx -\int_{B_r(x_0)} \tilde{h}w\varphi \,dx.
\end{align*}
Thus, \eqref{infinvahbilin2} implies that
\begin{equation} \label{infintransform}
\int_{B_{r}(x_0)} \langle A^T \nabla w  + w \mathbf{H}_1, \nabla \varphi \rangle dx  - \int_{B_r(x_0)} \left( \langle \mathbf{H}_2, \nabla w \rangle +\tilde{h} w \right) \varphi \,dx
= 0 \quad \text{ for all $\varphi \in H_0^{1,2}(B_{r}(x_0))$}.
\end{equation}
Since $\mathbf{H}_1 \in L^p(B_r(x_0), \mathbb{R}^d)$, $\mathbf{H}_2 \in L^d(B_r(x_0), \mathbb{R}^d)$ and $\tilde{h} \in L^{\tilde{q}}(B_r(x_0))$ with $\tilde{q} \in (\frac{d}{2}, \infty)$, it follows by \cite[Théorème 7.2]{S65} that $w$ has a continuous version in $B_r(x_0)$, say again $w \in H^{1,2}(B_r(x_0)) \cap C(B_r(x_0))$ (indeed, $w$ has a locally H\"{o}lder continuous version in $B_r(x_0)$). Moreover, it follows from \eqref{wpositivcon} that $w(x) \geq 0$ for all $x \in B_r(x_0)$. \\ \\
{\underline{\bf Claim}:}  $w(x) > 0$ for every $x \in B_{r}(x_0)$.  \\
To show the claim, we proceed by contradiction. Suppose there exists $y_0 \in B_{r}(x_0)$ such that $w(y_0) = 0$.  
Then, applying the Harnack inequality (see \cite[Théorème 8.1]{S65}) to \eqref{infintransform}, we deduce that $w$ must vanish identically on $B_R(x_0)$ for all $R \in (\|y_0 - x_0\|, r)$.  
Given that $R$ is arbitrary, it follows that $w = 0$ on $B_{r}(x_0)$, which implies $\mathcal{T}(w) = 0$ on $L^2(\partial B_r(x_0))$.
This, however, contradicts \eqref{tracewone}. Therefore, we conclude that our claim holds. \\
Let $x_1 \in U$. Since $w(x_1) > 0$, we define the normalized function $\rho \in H^{1,2}(B_r(x_0)) \cap C(B_r(x_0))$ by
\[
\rho(x) := \frac{1}{w(x_1)} w(x), \quad x \in B_{r}(x_0).
\]
Thus, \eqref{infinva} is fulfilled by \eqref{infinvahbilin2}. \\ \\
(ii)
Observe that by \eqref{infintransform},
\begin{equation} \label{infinvaharnack}
\int_{B_{r}(x_0)} \langle A^T \nabla \rho  + \rho \mathbf{H}_1, \nabla \varphi \rangle dx  - \int_{B_r(x_0)} \left( \langle \mathbf{H}_2, \nabla \rho \rangle +\tilde{h} \rho \right) \varphi \,dx
= 0 \quad \text{ for all $\varphi \in H_0^{1,2}(B_{r}(x_0))$}.
\end{equation}
Since $\rho(x_1)=1$, by applying the Harnack inequality (\cite[Théorème 8.1]{S65}) to \eqref{infinvaharnack}, the assertion follows.
\end{proof}

\begin{rem} \label{remcounter}
Whether the conclusion of Theorem \ref{existrho} can be derived under the assumption of {\bf (T1)} remains an open question. However, at the very least, our current proof method for Theorem \ref{existrho} is not sufficient to establish the result under assumption {\bf (T1)}.
The main difficulty arises from the fact that the assumption $\mathbf{H} \in L^d(B_r(x_0), \mathbb{R}^d)$ does not allow the solution to be locally bounded. To illustrate this point, consider $d \geq 3$ and the function
\begin{equation} \label{ourwfun}
w(x) := \frac{1}{\ln 2} \ln \left(1 + \frac{1}{\|x\|}\right), \quad x \in B_1:=\{ x \in \mathbb{R}^d : \|x\|<1 \}.
\end{equation}
Then, $w(x) >0$ for all $x \in B_1 \setminus \{ 0\}$, and $w \in H^{1,2}(B_1) \cap C(B_1 \setminus \{0 \})$.  
Moreover, we have
\[
\mathcal{T}(w) = 1 \quad \text{in } L^2(\partial B_1),
\]
where $\mathcal{T}: H^{1,2}_0(B_1) \rightarrow L^2(\partial B_1)$ is the trace operator as in \cite[Theorem 4.6]{EG15}.
Now define the vector field $\mathbf{H}: B_1 \to \mathbb{R}^d$ by
\[
\mathbf{H}(x) := -\nabla \ln w(x), \quad x \in B_1.
\]
Then, $\mathbf{H} \in L^d(B_1, \mathbb{R}^d)$, but $\mathbf{H} \notin \bigcup_{p\in (d, \infty)} L^p(B_1, \mathbb{R}^d)$.
Direct computation shows that $w$ satisfies \eqref{infinvahbilin2} with $B_r(x_0)$ replaced by $B_1$, and that $w$ is in fact the unique function satisfying both $\mathcal{T}(w) = 1$ and \eqref{infinvahbilin2}.
However, the function $w$ defined in \eqref{ourwfun} does not admit a locally bounded version in $B_1$.  
This demonstrates that the local boundedness of the solution cannot, in general, be deduced under the sole assumption $\mathbf{H} \in L^d(B_1, \mathbb{R}^d)$.
(If $\mathbf{H} \in L^p(B_1, \mathbb{R}^d)$ with $p \in (d, \infty)$, then the local boundedness of a solution follows by \cite[Theorem 5.1]{T73}. Indeed, one can check that $\operatorname{div}\mathbf{H} \in L^{\frac{d}{2}}(B_1)$, but $\operatorname{div}\mathbf{H} \notin \bigcup _{p \in (\frac{d}{2}, \infty)}L^{p}(B_1)$. Therefore, to obtain the local boundedness of solutions in case of $\mathbf{H} \in L^d(B_1,\mathbb{R}^d)$, the condition {\bf (T)} regarding $\mathbf{H}$ is essential.
\end{rem}
\centerline{}
Below, we present the core method of this paper, which transforms a general vector field into a divergence-free vector field. We shall refer to this as the divergence-free transformation. In particular, we present here a simplified form of \cite[Theorem 3.2]{L25jm}.
\begin{theo}{\bf (Divergence-free transformation)} \label{divfretransform}
Assume that {\bf (T)} holds. Let $\rho \in H^{1,2}(U) \cap C(\overline{U})$ be a strictly positive function on $\overline{U}$ constructed as in Theorem~\ref{existrho}. Define the vector field
\begin{equation} \label{divfreebdef}
\mathbf{B} := \mathbf{H} + \frac{1}{\rho} A^T \nabla \rho \quad \text{in } U.
\end{equation}
Then $\rho \mathbf{B} \in L^2(U, \mathbb{R}^d)$ and satisfies
\begin{equation} \label{divfeepropbolb}
\int_U \langle \rho \mathbf{B}, \nabla \varphi \rangle \, dx = 0 \quad \text{for all } \varphi \in C_0^\infty(U).
\end{equation}
Let $f \in L^1(U)$, and $u \in H^{1,2}_0(U)$ with $cu \in L^1(U)$. Then the following two statements are equivalent:
\begin{itemize}
  \item[(i)] The function $u$ satisfies \eqref{maineqfir}.
  \item[(ii)] The function $u$ satisfies
  \begin{equation*}
  \int_U \langle \rho A \nabla u, \nabla \varphi \rangle + \langle \rho \mathbf{B}, \nabla u \rangle \varphi + \rho c u \varphi \, dx = \int_U \rho f \varphi\, dx \quad \text{for all } \varphi \in C_0^\infty(U).
  \end{equation*}
\end{itemize}
In other words, $u$ is a weak solution to \eqref{maineq2}, if and only if $u$ is a weak solution to \eqref{maineq3}.
\end{theo}
\begin{proof}
The proof is identical to that of \cite[Theorem 3.2]{L25jm} in the case where $\mathbf{F} = 0$.
\end{proof}
\centerline{}
The following two lemmas, which play a supporting role in the proof of the main result, are adapted from \cite{L25jm} and \cite{L24}, respectively.
\begin{lem} \label{theomaineun}
Assume $d \geq 3$. Let $\hat{\lambda}>0$ be a constant, and let $\hat{A} = (\hat{a}_{ij})_{1 \leq i,j \leq d}$ be a matrix of bounded and measurable functions on $\mathbb{R}^d$ such that
\begin{equation} \label{ellipticihata}
\langle \hat{A}(x) \xi, \xi \rangle \geq \hat{\lambda} \| \xi \|^2 \quad \text{for a.e. } x \in \mathbb{R}^d \text{ and all } \xi \in \mathbb{R}^d.
\end{equation}
Let $\hat{\mathbf{B}} \in L^2(U, \mathbb{R}^d)$ be a vector field satisfying
\begin{equation} \label{weakdivfree}
\int_U \langle \hat{\mathbf{B}}, \nabla \varphi \rangle \, dx = 0 \quad \text{for all } \varphi \in C_0^\infty(U).
\end{equation}
Let $\hat{c} \in L^1(U)$ with $\hat{c} \geq 0$, and let $\hat{f} \in L^q(U)$ for some $q \in \left(\frac{d}{2}, \infty\right)$. Then, the following statements hold:
\begin{itemize}
\item[(i)] There exists a weak solution $\hat{u} \in H_0^{1,2}(U) \cap L^{\infty}(U)$ to
\begin{equation} \label{mainenerbd_nodiv}
\left\{
\begin{alignedat}{2}
- \operatorname{div}(\hat{A} \nabla \hat{u}) + \langle \hat{\mathbf{B}}, \nabla \hat{u} \rangle + \hat{c} \hat{u} &= \hat{f} && \quad \text{in } U, \\
\hat{u} &= 0 && \quad \text{on } \partial U,
\end{alignedat}
\right.
\end{equation}
i.e., $\hat{u} \in H_0^{1,2}(U)$ with $\hat{c} \hat{u} \in L^1(U)$ satisfies
\begin{equation*} 
\int_U \langle \hat{A} \nabla \hat{u}, \nabla \varphi \rangle + \langle \hat{\mathbf{B}}, \nabla \hat{u} \rangle \varphi + \hat{c} \hat{u} \varphi \, dx = \int_U \hat{f} \varphi \, dx \quad \text{for all } \varphi \in C_0^\infty(U).
\end{equation*}
Moreover, the following estimates hold:
\begin{align}
\|\hat{u}\|_{H_0^{1,2}(U)} &\leq \hat{K}_3 \| \hat{f} \|_{L^{\frac{2d}{d+2}}(U)}, \label{estimenbd_nodiv} \\
\|\hat{u}\|_{L^{\infty}(U)} &\leq \hat{K}_4 \| \hat{f} \|_{L^q(U)}, \label{bddestim_nodiv}
\end{align}
where $\hat{K}_3 > 0$ depends only on $d$, $\hat{\lambda}$, and $|U|$, and $\hat{K}_4 > 0$ depends only on $d$, $\hat{\lambda}$, $q$, and $|U|$.

\item[(ii)] Let $\hat{v} \in H_0^{1,2}(U)$ with $\hat{c} \hat{v} \in L^1(U)$ be such that
\begin{equation*} 
\int_U \langle \hat{A} \nabla \hat{v}, \nabla \varphi \rangle + \langle \hat{\mathbf{B}}, \nabla \hat{v} \rangle \varphi + \hat{c} \hat{v} \varphi \, dx = 0 \quad \text{for all } \varphi \in C_0^\infty(U).
\end{equation*}
Then $\hat{v} = 0$ in $U$. In particular, the solution $\hat{u}$ in (i) is unique.

\item[(iii)] Let $\alpha > 0$ and $\theta \in [1, \infty]$, and assume that $\hat{c} \geq \alpha$ and $\hat{f} \in L^\theta(U) \cap L^q(U)$. Then $\hat{u}$ in (i) satisfies
\begin{equation} \label{lrcontrapro}
\|\hat{u}\|_{L^\theta(U)} \leq \frac{1}{\alpha} \|\hat{f}\|_{L^\theta(U)}.
\end{equation}
\end{itemize}
\end{lem}
\begin{proof}
The assertion follows from \cite[Theorem 3.3]{L25jm} in the case where $\mathbf{F} = 0$.
\end{proof}

\begin{lem} \label{amsprocethm}
Assume $d \geq 3$. Let $\hat{\lambda}>0$ be a constant, and let $\hat{A} = (\hat{a}_{ij})_{1 \leq i,j \leq d}$ be a matrix of bounded and measurable functions on $\mathbb{R}^d$ satisfying \eqref{ellipticihata}. Let $\hat{\mathbf{B}} \in L^2(U, \mathbb{R}^d)$ be a vector field satisfying \eqref{weakdivfree}. Let $\hat{c} \in L^{\frac{2d}{d+2}}(U)$ with $\hat{c} \geq 0$, and let $\hat{f} \in L^{\frac{2d}{d+2}}(U)$. Then the following statements hold:
\begin{itemize}
\item[(i)]
There exists a unique solution $\hat{u}\in H^{1,2}_0(U)$ to \eqref{mainenerbd_nodiv}, and $\hat{u}$ satisfies the estimate \eqref{estimenbd_nodiv}.
\item[(ii)]
Let $\alpha > 0$ and $\theta \in [1, \infty]$, and assume that $\hat{c} \geq \alpha$ and $\hat{f} \in L^\theta(U) \cap L^{\frac{2d}{d+2}}(U)$. Then the solution $\hat{u}$ in (i) satisfies \eqref{lrcontrapro}.
\end{itemize}
\end{lem}
\begin{proof}
(i) The existence and uniqueness of the solution $\hat{u}$ to \eqref{mainenerbd_nodiv}, as well as the estimate \eqref{estimenbd_nodiv}, follow from \cite[Theorem 1.1(i)]{L24}. \\
(ii) The assertion follows from \cite[Theorem 1.1(ii)]{L24}.
\end{proof}
\centerline{}
Now, we present the proof of the main result stated in the Introduction.\\ \\
\noindent
{\bf Proof of Theorem \ref{mainexisthm}}\\
(i)  
Let $\rho \in H^{1,2}(U) \cap C(\overline{U})$ be a strictly positive function on $\overline{U}$ constructed as in Theorem~\ref{existrho}, and define the vector field $\mathbf{B}$ as in \eqref{divfreebdef}. Then \eqref{divfeepropbolb} is satisfied. Let $v \in H_0^{1,2}(U)$ with $cv \in L^1(U)$ be such that \eqref{weakdzero} holds.
By Theorem~\ref{divfretransform}, we obtain
\begin{equation*} 
\int_U \langle \rho A \nabla v, \nabla \varphi \rangle + \langle \rho \mathbf{B}, \nabla v \rangle \varphi + \rho c v \varphi \, dx = 0 \quad \text{for all } \varphi \in C_0^\infty(U).
\end{equation*}
Then, by Lemma~\ref{theomaineun}(ii), it follows that $v = 0$ in $U$. \\ 
(ii)  
Let $f \in L^q(U)$ for some $q \in (\frac{d}{2}, \infty)$. By Lemma~\ref{theomaineun}(i), there exists a unique function $u \in H_0^{1,2}(U) \cap L^{\infty}(U)$ satisfying
\begin{equation} \label{maineqfdivzesolu}
\int_U \langle \rho A \nabla u, \nabla \varphi \rangle + \langle \rho \mathbf{B}, \nabla u \rangle \varphi + \rho c u \varphi \, dx = \int_U \rho f \varphi \, dx \quad \text{for all } \varphi \in C_0^\infty(U),
\end{equation}
and \eqref{estimenbd_nodiv}, \eqref{bddestim_nodiv} and Theorem \ref{existrho}(ii)
imply
\begin{align}
\|u\|_{H_0^{1,2}(U)} &\leq K_3 \|f \rho\|_{L^{\frac{2d}{d+2}}(U)} \leq K_3 \max_{\overline{U}} \rho \cdot \|f\|_{L^{\frac{2d}{d+2}}(U)} \leq K_3 \tilde{K}_1 \|f\|_{L^{\frac{2d}{d+2}}(U)}, \label{energyestimgen} \\
\|u\|_{L^\infty(U)} &\leq K_4 \|f \rho\|_{L^q(U)} \leq K_4 \max_{\overline{U}} \rho \cdot \|f\|_{L^q(U)} \leq K_4 \tilde{K}_1 \|f\|_{L^q(U)}. \nonumber
\end{align}
By Theorem~\ref{divfretransform}, $u$ is a weak solution to \eqref{maineq2}, so that \eqref{estimenbd2}  and \eqref{bddestmizte} follow.
The uniqueness follows from (i). 
Note that since $\rho c \geq \alpha \min_{\overline{U}} \rho > 0$, we may apply Lemma~\ref{theomaineun}(iii) to \eqref{maineqfdivzesolu} to obtain
\[
\|u\|_{L^\theta(U)} \leq \frac{1}{\alpha \min_{\overline{U}} \rho} \|\rho f\|_{L^\theta(U)} 
\leq \frac{\max_{\overline{U}} \rho}{\alpha \min_{\overline{U}} \rho} \|f\|_{L^\theta(U)} 
\leq \frac{\tilde{K}_1}{\alpha} \|f\|_{L^\theta(U)}.
\]
\\
(iii)  
By Lemma~\ref{amsprocethm}(i), there exists $u \in H_0^{1,2}(U)$ satisfying both \eqref{maineqfdivzesolu} and the estimate \eqref{energyestimgen}. Again, by Theorem~\ref{divfretransform}, $u$ is a weak solution to \eqref{maineq2}, and the uniqueness follows from part (i).  Since $\rho c \geq \alpha \min_{\overline{U}} \rho > 0$, the contraction estimate follows from Lemma~\ref{amsprocethm}(ii). \\
\text{} \hfill $\square$ 
\centerline{}
The following provides an explicit example of a vector field $\mathbf{H} \in L^d(B_r(x_0), \mathbb{R}^d)$ that satisfies condition {\bf (T)} but does not belong to $\bigcup_{p \in (d, \infty)}L^p(B_r(x_0), \mathbb{R}^d)$.

\begin{exam}
Let $B_1:=\{x \in \mathbb{R}^d: \|x\|<1 \}$, and define $\Phi : B_1 \to \mathbb{R}$ by
\[
\Phi(x) := \ln \ln \left( 1 + \frac{1}{\|x\|} \right), \quad x \in B_1.
\]
Then $\nabla \Phi \in L^d(B_1, \mathbb{R}^d)$, but $\nabla \Phi \notin \bigcup_{p \in (d, \infty)} L^p(B_1, \mathbb{R}^d)$.  
By symmetry, for each $i \in \{1, \ldots, d\}$,
\[
\partial_i \Phi \in L^d(B_1), \quad \text{but} \quad \partial_i \Phi \notin \bigcup_{p \in (d, \infty)} L^p(B_1).
\]
Let $\mathbf{H}_1 \in L^p(B_1, \mathbb{R}^d)$ be an arbitrary vector field, and define $\mathbf{H}_2 : B_1 \to \mathbb{R}^d$ by
\[
\mathbf{H}_2 := (\partial_d \Phi, 0, \ldots, -\partial_1 \Phi) \quad \text{on } B_1.
\]
Then $\mathbf{H}_2 \in L^d(B_1, \mathbb{R}^d)$, but $\mathbf{H}_2 \notin \bigcup_{p \in (d, \infty)} L^p(B_1, \mathbb{R}^d)$.  
In particular, for all $\varphi \in C_0^\infty(B_1)$,
\[
\int_{B_1} \langle \mathbf{H}_2, \nabla \varphi \rangle\,dx = \int_{B_1} \partial_d \Phi \, \partial_1 \varphi - \partial_1 \Phi \, \partial_d \varphi \, dx = \int_{B_1} \Phi (-\partial_d \partial_1 \varphi + \partial_1 \partial_d \varphi) \, dx = 0,
\]
and hence $\operatorname{div}\mathbf{H}_2 = 0 \in L^{\tilde{q}}(B_1)$ for any $\tilde{q} \in (\frac{d}{2}, \infty)$.
Thus, the vector field $\mathbf{H} := \mathbf{H}_1 + \mathbf{H}_2 \in L^d(B_1, \mathbb{R}^d)$ satisfies condition {\bf (T)} but does not belong to $\bigcup_{p \in (d, \infty)} L^p(B_1, \mathbb{R}^d)$.
\end{exam}

\section{Conclusions and discussion} \label{sec5}

This paper establishes the existence and uniqueness of weak solutions to homogeneous boundary value problems for linear elliptic equations with drift coefficients $\mathbf{H} \in L^d(U, \mathbb{R}^d)$, under the assumption that $\mathbf{H}$ satisfies a suitable divergence-type condition. The argument fundamentally relies on the elliptic regularity results (H\"{o}lder regularity and the Harnack inequality) of G.~Stampacchia \cite{S65}, which remain applicable even in the critical case. A key analytical observation is that both the Harnack inequality and H\"{o}lder continuity hold despite the limited regularity of the drift term. \\
In contrast to the framework developed in \cite{L25jm}, the present work does not provide quantitative control over the constants appearing in the a priori estimates. For instance, the constant $\tilde{K}_1 \geq 1$ in Theorem~\ref{existrho} depends on the drift $\mathbf{H}$ itself rather than its norm in a specific function space. It remains unclear whether such constants remain stable under mollification or other approximation procedures for $\mathbf{H}$, and further investigation is needed to address this issue.\\
Another natural question is whether the results extend beyond the critical case $\mathbf{H} \in L^d(U, \mathbb{R}^d)$ to the subcritical setting $\mathbf{H} \in L^2(U, \mathbb{R}^d)$. Although some special cases have been studied, such as divergence-free drifts \cite{K07} and drifts with nonnegative divergence \cite{L24}, the general case with drifts in $L^2$ or $L^d$ remains open. Addressing this problem would likely require a more delicate analysis.\\
Finally, the methods developed in this paper are not confined to the context of linear divergence-form equations. They may also be applicable to regularity theory for double-divergence form equations and to the study of invariant measures for stochastic analysis, as in \cite{L25bv}. \\


\subsection*{\bf Acknowledgments}
This work was supported by the National Research Foundation of Korea (NRF) grant funded by the Korea government (MSIT) (RS-2025-16070171).


\centerline{}
\centerline{}
\noindent
Haesung Lee\\
Department of Mathematics and Big Data Science,  \\
Kumoh National Institute of Technology, \\
Gumi, Gyeongsangbuk-do 39177, Republic of Korea, \\
E-mail: fthslt@kumoh.ac.kr, \; fthslt14@gmail.com
\end{document}